\documentclass[12pt,reqno]{amsart}
\usepackage[includehead,includefoot,margin=19mm]{geometry}
\usepackage[latin1]{inputenc}
\usepackage{amssymb}
\usepackage{amsmath}
\usepackage{amsfonts}
\usepackage{mathrsfs}
\usepackage{hyperref}
\usepackage[all]{xy}  
\usepackage{verbatim}
 \usepackage{graphicx}

\theoremstyle{plain}
\newtheorem{theorem}{Theorem}[section]

\newtheorem{lemma}[theorem]{Lemma}
\newtheorem{proposition}[theorem]{Proposition}

\newtheorem{corollary}[theorem]{Corollary}
\newtheorem*{theorem*}{Theorem}

\theoremstyle{definition}
\newtheorem{definition}[theorem]{Definition}

\newtheorem{example}[theorem]{Example}

\theoremstyle{remark}

\newtheorem{remark}[theorem]{Remark}

\def\ZZ{{\mathbb Z}}

\def\QQ{{\mathbb Q}}
\def\CC{{\mathbb C}}

\def\TT{{\mathbb{T}}}

\def\dim{{\rm dim}}
\def\codim{{\rm codim}}
\def\deg{{\rm deg}}

\def\P{{\mathbb{P}}}
\def\AA{{\mathbb{A}}}
\def\spec{{\rm Spec}}

\def\L{{\mathscr{L}}}

\def\A{{\mathcal{A}}}
\def\D{{\mathfrak{D}}}
\def\Ca{{\rm CaDiv}}
\def\X{V}
\def\Z{Z}
\def\CS{{\rm CS}}
\def\Orb{\mathcal{O}}
\def\H{\mathcal{H}}
\def\lim{{\rm lim}}
\def\L{\mathscr{L}}
\def\supp{{\rm Supp}}

\def\EE{\mathscr{E}}
\def\cone{{\rm Cone}}
\def\pos{\mathscr{P}}
\def\int{{\rm Int}}
\def\inc{\mathscr{A}}

\def\PPDiv{{\rm PPDiv}_{\QQ}}

\author{Marta Agustin Vicente}
\address{Basque Center of Applied Mathematics (BCAM)\\  Mazarredo, 14. 48009 Bilbao \\ Spain}
\email{martaav22@gmail.com} 
\author{Kevin Langlois}
\address{Departemento de Matem\'atica \\ Universidade Federal do Cear\'a (UFC) \\ Campus do Pici, Bloco 914, CEP 60455-760. Fortaleza-Ce, Brasil}
\email{kevin.langlois@mat.ufc.br} 
\title[Intersection cohomology of $\TT$-varieties]
{Decomposition theorem and  torus \\ actions of complexity one}

\begin{document}
\begin{abstract}
We algorithmically compute the intersection cohomology Betti numbers of any complete normal algebraic variety with a torus action of complexity one.
\end{abstract}

\maketitle

\setcounter{tocdepth}{1}
\tableofcontents
\section{Introduction}
This paper investigates the intersection cohomology for complete normal algebraic varieties equipped with a \emph{complexity-one torus action} (i.e., a torus action whose general orbits are of codimension one). We algorithmically compute the intersection cohomology Betti numbers of these varieties (see Theorem \ref{t-hvector}) in the language of polyhedral divisors. In particular, our main result generalizes the classical case of complete toric varieties treated in \cite{Sta87, DL91, Fie91}.

Let us introduce some notation in order to formulate and explain our results.
We will be working with a complex torus $\TT= (\CC^{\star})^{n}$ and a normal algebraic variety $X$
over $\CC$  endowed with an algebraic faithful $\TT$-action  of complexity one.
To perform our computation, we use in this article the geometric and combinatorial approach developed in \cite{Tim08, AH06, AHS08, AIPSV12}, which extends into our setting the well-known classification
of toric varieties by their defining fans. 

More precisely, this description 
can be briefly explained as follows (see Section \ref{sec-r-tvar} for a detailed exposition). 
As a first ingredient, we need to consider the smooth projective curve $Y$ parameterizing the general $\TT$-orbits, i.e., the curve $Y$ comes with an invariant rational map $\iota:X\dashrightarrow Y$ giving rise to an identification of function fields
$\CC(Y)\simeq\CC(X)^{\TT}$. Then, one remarks that the local structure of the torus action restricts to the affine case (see \cite[Section 3, Corollary 2]{Sum74}), namely, there exists a finite covering $(X_{i})_{i\in I}$ of $X$ by affine $\TT$-stable Zariski open subsets. The main result in \cite{AH06} states that 
the coordinate ring of the affine chart $X_{i}$ is built from
 a graded sheaf of $\mathcal{O}_{Y_{i}}$-algebras 
$$\A_{i} = \bigoplus_{u\in \sigma^{\vee}_{i}\cap M}\mathcal{O}_{Y_{i}}(\D^{i}(u)),$$
where $\sigma^{\vee}_{i}$ is a full-dimensional polyhedral cone living in the $\QQ$-vector space associated with the character lattice $M\simeq \ZZ^{n}$ of the torus $\TT$.
Moreover, the variety $Y_{i}$ is a dense Zariski open subset of $Y$ and the arrow
$$\sigma^{\vee}_{i}\rightarrow \Ca(Y_{i}),\,u\mapsto \D^{i}(u)$$
is a certain piecewise linear map (called a \emph{polyhedral divisor}) sending a vector to a $\QQ$-Cartier divisor over $Y_{i}$ and satisfying some positivity conditions (see \cite[Section 2]{AHS08} for more information). Finally, the $\TT$-variety $X_{i}$ is then identified with the global spectrum of the $M$-graded ring $\Gamma(Y_{i}, \A_{i})$.

One may additionally  choose the $\D^{i}$'s in a such way 
that the relative spectra $\spec_{Y_{i}}\, \A_{i}$ glue together into a normal $\TT$-variety $\tilde{X}$.
The resulting morphism $\pi:\tilde{X}\rightarrow X$ is called the \emph{contraction map}. It is known that $\pi$ is a proper birational morphism and that the total space $\tilde{X}$ only depends  on $X$, and neither on the choice of the open covering, or on the choice of the sheaves $\A_{i}$. Also the quotient map $\gamma: \tilde{X}\rightarrow Y_{0}$ onto a dense open subset $Y_{0}\subseteq Y$ yields a resolution of the indeterminacy locus of  the rational map $\iota:X\dashrightarrow Y$. 
 More precisely, Vollmert proved that the contraction space $\tilde{X}$ is in fact the normalization of the graph of the rational quotient map 
$X\dashrightarrow Y$ (compare \cite[Section 3, Lemma 1]{Vol10} and see also \cite{CP13}).

Note that the above geometric and combinatorial description can be reversed by defining a normal $\TT$-variety
of complexity one via a finite set of compatible polyhedral divisors (called a \emph{divisorial fan}, see \cite[Section 5]{AHS08}).
It is worthwhile mentioning that the  variety $\tilde{X}$ is toroidal in the 
sense of \cite[Chapter IV]{KKMS73}. Thus, we may expect to compute the intersection
cohomology of $\tilde{X}$ by using toric methods. Based on the work of de Cataldo, Migliorini, and Musta\c{t}\u{a} \cite{CMM15}, this case was done 
in \cite{AL17} (see also \ref{r-contfree}). Therefore, the next step consists to relate the intersection cohomology of $X$ and $\tilde{X}$ via the decomposition theorem of  Beilinson, Bernstein, Deligne and Gabber, see \cite[Theorem 6.2.5]{BBD82}. 

In this article,  we complete this step by providing  (after some reductions) an explicit description of the intersection cohomology of $X$ in terms of the one of $\tilde{X}$ and the (shifted) simple perverse sheaves appearing in the decomposition theorem of the map $\pi$ (combine Theorem \ref{t-main} and Theorem \ref{t-hvector}). 

For a $d$-dimensional complex algebraic variety $Z$,
we denote by $IC_{Z}$ the intersection cohomology complex with respect to the middle perversity (see Section \ref{sec-ic-conv} for more details) and by $IH^{i}(Z; \QQ)$ the intersection cohomology groups. Recall that $Z$ is a \emph{rationally smooth} variety if
for any $x\in Z$, we have that $H^{d}_{x}(Z;\QQ) = \QQ$ and $H^{i}_{x}(Z; \QQ)= 0$ for $i\neq d$, where $H^{\star}_{x}(Z;\QQ) = H^{\star}(Z, Z\setminus\{x\};\QQ)$ is the rational cohomology with support in $\{x\}$. Remark that rationally smooth varieties include quotients of smooth algebraic varieties by finite group actions (see \cite[Proposition A1]{Bri99}). 

Our first result below asserts that (up to a finite group action)  the local systems in the decomposition theorem for the map $\pi$ are trivial and supported on certain torus orbits.
\begin{theorem}\label{t-main}
Let $X$ be a normal complex algebraic variety with a torus action of complexity one.
Let $\pi:\tilde{X}\rightarrow X$ be the contraction map. Denote by $E\subseteq X$ the image of the exceptional locus of $\pi$ and let $\Orb(E)$ be the set of torus orbits in $E$. Then the following assertions hold.
\begin{itemize}
\item[(i)] There exists a finite subgroup $G\subseteq \TT$ such that we have a commutative diagram
 $$\xymatrix{
      \tilde{X}\ar[r]^{\pi} \ar[d]  & X \ar[d]\\ \tilde{X}/G \ar[r]^{\pi_{G}} & X/G,}$$
where the horizontal arrows $\pi$ and $\pi_{G}$ are contraction maps and the vertical ones are the quotients by the $G$-actions,
and with the further condition that
the image $E/G$ of the exceptional locus of $\pi_{G}$ have connected stabilizers. Moreover, we have the $\QQ$-vector space isomorphims 
$$IH^{j}(X; \QQ) \simeq IH^{j} (X/G; \QQ) \text{ and } IH^{j}(\tilde{X}; \QQ) \simeq IH^{j} (\tilde{X}/G; \QQ) \text{ for any }j\in \ZZ.$$ 
\item[(ii)] If the stabilizers of the points of  $E$ are connected, then there is an isomorphism
$$\pi_{\star}IC_{\tilde{X}}\simeq IC_{X}\oplus\bigoplus_{O\in\Orb(E)}\bigoplus_{b\in\ZZ}(\iota_{O})_{\star}IC_{\bar{O}}^{\oplus s_{b,O}}[-b],$$ 
where the $s_{b,O}\in \ZZ_{\geq 0}$ are  $0$ for all but finitely many $b\in\ZZ$ and $O\in\Orb(E)$. Here $\bar{O}$ is the Zariski closure of the orbit $O$ and $\iota_{O}:\bar{O}\rightarrow X$ is the natural inclusion.
\item[(iii)]  If $\tilde{X}$ is rationally smooth, then 
$$\pi_{\star}IC_{\tilde{X}}\simeq IC_{X}\oplus\bigoplus_{O\in\Orb_{2}(E)}(\iota_{O})_{\star}IC_{\bar{O}},$$
where $\Orb_{2}(E)$ is the set of codimension-two torus orbits in $E$.  
\end{itemize}
\end{theorem}
The presence of the finite group $G$ in Theorem \ref{t-main} $(i)$ simply means that we are changing the ambient lattice $M$ 
in the combinatorics of polyhedral divisors by a smaller one of a finite index. Geometrically, the group $G$ can be chosen as the subgroup
of $\TT$ generated by the groups of connected components of the stabilizers $\TT_{x}$, where $x$ runs over $E$ (this generates a finite group 
since $E$ has finitely many orbits, see \ref{l-compl}).
 Moreover, note that Statement $(i)$ in Theorem \ref{t-main} is a reduction of our problem of computing the intersection cohomology Betti numbers to the case where the points of $E$ have connected stabilizers. This technical assumption on the connectedness of the stabilizers will be taken for the rest of the introduction when we will speak about the decomposition theorem for the map $\pi$.

We now  describe the content of each section and explain how to calculate the intersection cohomology Betti numbers of any complete normal variety with torus action of complexity one. Section \ref{sec-prem} is devoted to introduce the necessary backgrounds on $\TT$-varieties and on intersection cohomology. We also introduce, in this section, the concept of Seifert torus bundles
which plays a key role in many places of the paper. 
In Section \ref{s-one}, we exhibit a natural stratification that makes $\pi$ a topological fibration over each stratum (see Lemma \ref{l-stratum}). We then show that $\pi$ is a semi-small map (see Lemma \ref{l-semi}). 
We prove Theorem \ref{t-main}  in Section \ref{s-two}. 

As we can see, the third part of the statement of Theorem \ref{t-main} is more explicit.
This is due to the decomposition theorem for semi-small maps (see for instance \cite[Theorem 8.2.36]{HTT08}). This result applies a priori
in the case where the variety $\tilde{X}$ is rationally smooth (e.g., when $\tilde{X}$ is a surface). But most of the time, the variety $\tilde{X}$ does not satisfy this assumption.
Thus, in order to bypass the difficulty of precisely determining the multiplicities $s_{b,O}$ of Theorem \ref{t-main} $(ii)$, in  Section \ref{s-three}, we study the stalks of the intersection cohomology sheaves on both sides of the decomposition theorem for the map $\pi$. This determines the intersection cohomology groups of every affine normal $\TT$-variety of complexity one with a unique attractive fixed point (see Proposition \ref{t-fixedpointt} and Example \ref{e-elliptic}) and the ones of a relative spectrum
associated with a polyhedral divisor having full-dimensional tail cone (see Theorem \ref{g-vectorth} and Section \ref{sec-r-tvar} for the notion of sheaf associated with a polyhedral divisor). 

Note that in order to prove Theorem \ref{g-vectorth}, we use results in \cite{BM99} which rely on a sheaf realization of the relative $g$-invariants of an affine toric variety.  
Inspired by the techniques developed in \cite[Section 7]{CMM15},
we finally sum up all our computations in Theorem \ref{t-hvector} where 
the $s_{b,O}$'s are obtained via an induction process on the dimension and from an inversion formula in the incidence algebra of the poset of orbits of $E$. 

As an application,
 we treat the cases of  dimensions $2$ and $3$.
For surfaces, we recover the work of Fieseler and Kaup \cite{FK86} (for rational coefficients), see Example \ref{ex-surf} and Example \ref{e-surface2}. In their work, they
describe the intersection cohomology Betti numbers of a complete normal $\CC^{\star}$-surface in terms of the set of elliptic fixed points. This set is in fact  the image of the exceptional locus of the contraction map. Thus, the description of the  intersection cohomology for torus actions of complexity one  in terms of the geometry of the contraction map might be seen as a generalization of \cite{FK86} in higher dimension.  For threefolds, we obtain the following result (see Theorem \ref{theore-dim3}). We recall that the \emph{Poincar\'e polynomial} of a complex algebraic variety $Z$ is the polynomial $P_{Z}(t) = \sum_{i = 1}^{2d}b_{i}(Z)t^{i}$, where $d = \dim\, Z$ and $b_{i}(Z) =  \dim\, IH^{i}(Z; \QQ)$. 
\begin{theorem}
Let $X$ be a complete normal threefold with a faithful action of $\TT = ( \CC^{\star})^{2}$, and with defining divisorial fan
$\EE$ over the smooth projective curve $Y$. Let $E$ be the image of the exceptional locus of the contraction map $\pi: \tilde{X}\rightarrow X$. Then
the Poincar\'e polynomial $P_{X}(t)$ is given by the formula
$$P_{X}(t) = ((1-r)t^{2} + 2\rho_{g}(Y)t + 1-r)(t^{4} + (\delta(\Sigma(\tilde{\EE}))-2)t^{2} + 1)$$ 
$$+ \sum_{y\in \supp(\tilde{\EE})}(t^{6} + (\delta(\tilde{\EE}_{y}) - 3)t^{4} + (\delta(\tilde{\EE}_{y}) - 3)t^{2} + 1) - |\mathcal{O}_{2}(E)|(t^{2} +1)t^{2},$$
where for a fan $\Sigma$ we denote by $\delta(\Sigma)$ the number of its rays and $\rho_{g}(Y)$ is the genus of the curve $Y$. 
Moreover, $ \Sigma(\tilde{\EE})$ (respectively, $\supp(\tilde{\EE})$) is the fan of the general fiber (respectively, the set of the points of $Y$ that give rise to 
a special fiber) of the quotient map $$\gamma: \tilde{X}\rightarrow Y\text{ and }r  = |\supp(\tilde{\EE})|.$$ Finally, the symbol $\tilde{\EE}_{y}$ stands for a $3$-dimensional fan describing the fiber of $\gamma$ at $y$.
\end{theorem}

Our approach in higher dimension does not give a direct formula of the intersection cohomology Betti numbers but provides an algorithm. This was somehow philosophically expected.  For instance, in the toric setting (see  \cite{DL91, Fie91}), the intersection cohomology Betti numbers are usually obtained in terms of a double induction, which compares
the global and local intersection cohomologies of toric varieties of smaller dimension. 
\\

{\em Related works.} Independently,  Laface, Liendo, and Moraga investigated in \cite{LLM17} the Hodge-Deligne polynomial, the cohomology ring, and the Chow ring of a smooth complete complexity-one $\TT$-variety. In relation with mirror symmetry, one may also look at the \emph{stringy invariants} (see \cite{BD96}) as studied
by Batyrev and Moreau in \cite{BM13}. Such a description (see \cite{Lan17,LPR19}) is available for complexity-one horospherical varieties \cite{LT16, LT17}. Finally, we want to mention that the starting point of the present work comes from the relative version of the $h$-invariant theory
for toric varieties, see \cite{KS16, CMM15, Cat15}.
\\

\begin{remark} 
In this paper, we work over the complex numbers and with the Euclidean topology. In particular, \emph{local systems} are locally constant sheaves of $\QQ$-vector spaces with respect to the Euclidean topology and whose fibers are finite-dimensional. 
\end{remark}

\section{Preliminaries}\label{sec-prem}
\subsection{Preliminaries on $\TT$-varieties}\label{sec-r-tvar}
This section is a reminder on the Altmann-Hausen's theory for the classification of torus actions on normal algebraic varieties  \cite{AH06, AHS08, AIPSV12, Lan15}. Here we focus on the complexity-one torus actions and we freely use the approach of Timashev in \cite{Tim08}. We 
also refer the reader to \cite{Tim97, LT16, LT17, Lan16} for generalizations to other reductive group actions. 
\\

We start by introducing the usual notation from toric geometry. Let $\TT = (\CC^{\star})^{n}$ be an algebraic torus with character lattice $M$ and one-parameter subgroup lattice $N$. Throughout this article, we will consider the duality between the lattices $M$ and $N$, and its extension
$$M_{\QQ}\times N_{\QQ}\rightarrow \QQ, \,\, (m, v)\mapsto \langle m, v \rangle,$$
to the $\QQ$-vector spaces $M_{\QQ} := \QQ\otimes_{\ZZ} M$ and $N_{\QQ} := \QQ\otimes_{\ZZ} N$. Recall that a polyhedral cone $\sigma\subseteq N_{\QQ}$ is said to be \emph{strictly convex} if $\{0\}\subseteq \sigma$ is  a $0$-dimensional face. This last condition exactly means that 
the \emph{dual cone} 
$$\sigma^{\vee}= \{ m\in M_{\QQ}\, |\, \forall v\in \sigma, \, \langle m, v \rangle \geq 0\}$$
is full-dimensional.

A finite set $\Sigma$ of strictly convex polyhedral cones  of $N_{\QQ}$ is a \emph{fan} if it is stable by taking the faces of its elements and if the intersection of any two elements  is a common face of both. A variety $X$ with a $\TT$-action is \emph{toric} (for the torus $\TT$) if $X$ is normal and if the torus $\TT$ faithfully acts on $X$ with an open dense orbit. Actually, the notion of fans mirrors the notion of toric varieties. Indeed,  any face $\tau$ of $\sigma\in \Sigma$ naturally defines an open immersion $X_{\sigma}\rightarrow X_{\tau}$, where  $X_{\sigma} = \spec\, \CC[\sigma^{\vee}\cap M]$  stands for the spectrum of the semigroup algebra of $\sigma^{\vee}\cap M$. From this, one may attach to any fan $\Sigma$ 
a toric variety  $X_{\Sigma}$ (and vice versa) by gluing all the affine pieces $X_{\sigma}$
from the elements of $\Sigma$ (see \cite[Chapter 3]{CLS11} for more details).
\\

We now pass to the notion of polyhedral divisors, which are the combinatorial objects that classify, in the affine case, the complexity-one torus actions. Let us fix a strictly convex polyhedral cone $\sigma\subseteq N_{\QQ}$. 
A \emph{$\sigma$-polyhedron} of $N_{\QQ}$ is a subset of $N_{\QQ}$ obtained as the Minkowski sum $Q+\sigma$ of a \emph{polytope} $Q\subseteq N_{\QQ}$ (i.e., the convex hull of a non-empty finite set of $N_{\QQ}$) and the cone $\sigma$. For a smooth complex algebraic curve $Y$, a \emph{$\sigma$-polyhedral divisor} $\D$ over $Y$ will be a formal sum
$$ \D  =  \sum_{y\in Y} \D_{y}\cdot [y], $$
where each $\D_{y}$ is a $\sigma$-polyhedron of $N_{\QQ}$ and $\D_{y}=\sigma$ for all but finitely many $y\in Y$. For every $m\in \sigma^{\vee}$ we define the $\QQ$-divisor $\D(m)$ called the \emph{evaluation} of $\D$ at $m$ by setting 
$$\D(m) = \sum_{y\in Y} \min_{v\in \D_{y}}\, \langle m , v\rangle \cdot [y].$$
The curve $Y$ (respectively, the cone $\sigma$) is called the \emph{locus} (respectively, the \emph{tail}) of $\D$. The $\sigma$-polyhedral divisor $\D$ is said to be \emph{proper} (or \emph{positive}) if $Y$ is affine, or $Y$ is projective and the two the following properties are satisfied.
\begin{itemize}
\item[(i)] First, the \emph{degree} $\deg(\D) = \sum_{y\in Y} \D_{y}$ must be strictly contained in $\sigma$.
\item[(ii)] Secondly, for every $m\in \sigma^{\vee}$ such that $\min_{v\in \deg(\D)}\langle m, v \rangle = 0$, the 
divisor $\D(dm)$ has to be principal for some $d\in \ZZ_{>0}$.
\end{itemize}
In the sequel, we write $\PPDiv(Y, \sigma)$ (or $\PPDiv(Y, N, \sigma)$ if one needs to emphasize on the lattice $N$)  for the set of proper $\sigma$-polyhedral divisors on $Y$.
 If $Y$ is affine,
then we make the convention that the degree $\deg(\D)$ is empty. 
\\

Note that the functor $X\mapsto \CC[X] =  \Gamma(X, \mathcal{O}_{X})$
induces an equivalence  between the category of affine varieties with $\TT$-action (and with $\TT$-equivariant morphisms), and the opposite category of integral finite type $\CC$-algebras with an $M$-grading (together with homogeneous $\CC$-algebra morphisms). 
The following theorem, which is \cite[Theorems 3.1, 3.4]{AH06} applied to complexity-one torus actions, makes more precise this correspondence. 
\begin{theorem}
\begin{itemize}
\item[(i)] Let $\sigma\subseteq N_{\QQ}$ be a strictly convex polyhedral cone and let $Y$ be a smooth complex algebraic curve. If $\D\in\PPDiv(Y, \sigma)$, then the $M$-graded subalgebra
$$A(Y, \D) := \bigoplus_{m\in \sigma^{\vee}\cap M} H^{0}(Y, \mathcal{O}_{Y}(\D(m)))\otimes \chi^{m}\subseteq \CC(Y)\otimes_{\CC}\CC[\TT],$$
where $\chi^{m}$ is the Laurent monomial corresponding to $m\in M$, defines a normal affine $\TT$-variety $X(\D) = X(Y, \D):= \spec\, A(Y, \D)$ of complexity one with rational quotient $Y$.
\item[(ii)] Conversely, if $X$ is a normal affine variety with a faithful $\TT$-action of complexity one, then there exists a strictly convex polyhedral cone $\sigma\subseteq N_{\QQ}$, a  smooth complex algebraic curve $Y$ and $\D\in \PPDiv(Y, \sigma)$ such that the $\TT$-variety $X(\D)$ is equivariantly isomorphic to $X$. 
\end{itemize}
\end{theorem}
Divisorial fans generalize polyhedral divisors in the same way as in toric geometry, when one passes from the concept of strictly convex cones to the concept of fans.
\begin{definition}
Let $Y$ be a smooth complex curve. A \emph{divisorial fan} on $(Y, N)$  is a finite set 
$$\EE =  \{ \D^{i}\, |\, i\in I\}$$
with $\D^{i}\in \PPDiv(Y_{i}, \sigma_{i})$, where $Y_{i}\subseteq Y$ is a Zariski dense open subset and $\sigma_{i}\subseteq N_{\QQ}$ is a strictly convex polyhedral cone, subject to the following conditions. 
\begin{itemize}
\item[(i)] \emph{(stability by intersection)} We have 
$$\D^{i}\cap \D^{j} := \sum_{y\in  Y_{ij}} (\D_{y}^{i} \cap \D_{y}^{j})\cdot [y]\in \EE$$
 for all $i, j\in I$, where
$$ Y_{ij}  := \{ y\in Y_{i}\cap Y_{j}\,|\, \D_{y}^{i}\cap \D_{y}^{j}\neq  \emptyset\}\subseteq Y.$$
\item[(ii)]\emph{(face relations)} For all $i, j\in I$ and for any $y \in Y_{ij}$ the polyhedron $\D_{y}^{i} \cap \D_{y}^{j}$ is a common
face of $\D_{y}^{i}$ and $\D_{y}^{j}$. 
\item[(iii)]\emph{(degree relations)} We have $\deg(\D^{i} \cap \D^{j})=  \deg(\D^{j})\cap \sigma_{i}\cap\sigma_{j}$
for all $i,j\in I$ (note that $\deg(\D^{i}) =  \emptyset$ whenever the locus of $\D^{i}$ is affine).  
\item[(iv)] We will also ask that $Y =\bigcup_{i\in I}Y_{i}$.
\end{itemize}
\end{definition}

The definition of a divisorial fan implies the existence  of transition maps
$$X(\D^{i})\leftarrow X(\D^{i}\cap \D^{j})\rightarrow X(\D^{j})$$
that are open immersions. These maps allow to glue together the open sets $X(\D^{i})$ into a normal $\TT$-variety $X(\EE)$ of complexity one (see \cite[Sections 3,4]{AHS08}).
Note that inside $X(\EE)$ the subset $X(\D^{i}\cap \D^{j})$ identifies with the intersection $X(\D^{i})\cap X(\D^{j})$
 for all $i, j\in I$. Conversely, every normal algebraic variety with torus action of complexity one arises from a divisorial fan (see \cite[Theorem 5.6]{AHS08}).
\\

Any $\sigma$-polyhedral divisor $\D$ over $Y$ yields a sheaf of $\mathcal{O}_{Y}$-algebras
$$\mathcal{A} =  \mathcal{A}(\D)  =  \bigoplus_{m\in \sigma^{\vee}\cap M} \mathcal{O}_{Y}(\D(m)).$$
Now given a divisorial fan $\EE$ describing the $\TT$-variety $X=  X(\EE)$, the \emph{contraction map} $\pi: \tilde{X}\rightarrow X$ is defined by
gluing the natural morphisms  $$\spec_{Y_{i}}\mathcal{A}(\D^{i})\rightarrow \spec\, \Gamma(Y_{i}, \mathcal{A}(\D^{i})) \text{ for all }\D^{i}\in \EE.$$  It follows that the 
total space $\tilde{X}$ is identified with $X(\tilde{\EE})$, where $\tilde{\EE}$ is the divisorial fan 
$$\tilde{\EE} := \left\{\D_{|U_{j}}:= \sum_{y\in U_{j}}\D_{y}\cdot [y]\, |\, \D\in \EE\text{ and } j\in J\right\}$$
and $(U_{j})_{j\in J}$ is any finite Zariski open affine covering of $Y$. Moreover,  by \cite[Section 3, Lemma 1]{Vol10}, one sees that the pair $(\tilde{X}, \pi)$ (up to composition with an equivariant isomorphism) depends only on the $\TT$-variety $X(\EE)$. We will say that
the $\TT$-variety $X(\EE)$ is \emph{contraction-free} if the rational quotient map $X(\EE)\dashrightarrow Y$ is globally defined, or equivalently, if the contraction map $\pi$ is an isomorphism.  
\begin{remark}
 One can recover some geometric properties of the normal $\TT$-variety $X(\EE)$ in terms of the combinatorics of its divisorial fan $\EE$. For instance,  the $\TT$-variety $X(\EE)$ is complete (respectively, contraction-free) if and only if $Y$ is complete and
for any $y\in Y$ the sequence of polyhedral coefficients of all $\D^{i}\in \EE$ at $y$ covers $N_{\QQ}$ \cite[Section 7]{AHS08} (respectively, the locus of each element of $\EE$ is affine).
\end{remark}
We finally discuss on some combinatorial objects that come from the description of Timashev for torus actions of complexity one (see \cite{Tim08}).
\\

Let $\EE$ be a divisorial fan over $(Y, N)$ and assume that $Y$ is a smooth complex projective curve.
 Let $\mathscr{N}$ be the quotient set $$Y \times N_{\QQ}\times \QQ_{\geq 0}/\sim,$$
 where the equivalence relation $\sim$ is given by $$(y, a, b)\sim (y', a', b')\text{ if and only if }(y = y', a = a', b =b')\text{ or }(a = a', b=b' =0).$$ The set $\mathscr{N}$ is 
usually called the \emph{hyperspace} (see \cite[Section 16.2]{Tim11}).   Note that the natural map $$N_{\QQ}\rightarrow \mathscr{N}, \,\, a\mapsto (y, a, 0)/\sim$$ allows to identify $N_{\QQ}$ with a subset of $\mathscr{N}$. Inside $\mathscr{N}$, one considers the
\emph{hypercone} 
$$C(\D) : = \left(\,\,\bigcup_{y\in Y}\{y\}\times C_{y}(\D)\,\,\right)/\sim$$
for any $\sigma$-polyhedral divisor $\D\in\EE$. Here the subset $C_{y}(\D)\subseteq N_{\QQ}\times \QQ$ is the \emph{Cayley cone} of $\D$ at the point $y$, that is, the cone generated by $(\sigma\times \{0\})\cup (\D_{y}\times \{1\})$.
\\

 For each $\D\in \EE$ defined over the curve $Y$, a \emph{hyperface} of $C(\D)$ is a subset of the form $C(\D')\subseteq C(\D)$
for $\D'$ a polyhedral divisor over $Y$ and such that $C_{y}(\D')$ is 
a face of $C_{y}(\D)$ for any $y\in Y$.  We set 
$$HF(\EE):= \{ C\text{ hyperfaces of } C(\D)\, |\, \D\in \EE\text{ and } C\cap \deg(\D)\neq \emptyset\}$$
and endow it  with a  structure of poset by considering the hyperface relations. Observe that
$HF(\EE)$ is isomorphic to the concrete poset
$$\{\tau\text{ face of }\sigma\, |\, \sigma \text{ tail of } \D, \, \D\in \EE \text{ and } \tau\cap \deg(\D)\neq \emptyset\}$$
together with the face relations.
\\

The poset  $HF(\EE)$ has the following geometric meaning. Consider the contraction map $$\pi: \tilde{X} =  X(\tilde{\EE})\rightarrow X =  X(\EE).$$ Denote by $E$ the image of the exceptional locus of the map $\pi$ and let  $\mathcal{O}(E)$ be the set of $\TT$-orbits of $E$. Then the relation $\prec$, where  two orbits $O_{1}, O_{2}$ satisfy $O_{1}\prec O_{2}$ if and only if $O_{2}\subseteq \bar{O}_{1}$, makes $\mathcal{O}(E)$ a poset which is isomorphic to $HF(\EE)$. This follows from the results in \cite[Section 16.4]{Tim11}.
The correspondence comes by seeing that an element $(y, a,b)/\sim$ of $\mathscr{N}$ naturally defines a $\TT$-invariant discrete valuation ${\rm val}_{y, a, b}$ on the function field $\CC(X)$ via the relation 
$${\rm val}_{y, a, b}(f\otimes \chi^{m}) = \langle m , a \rangle + b \, {\rm ord}_{y}(f),$$
where $f\in \CC(Y)^{\star}$ and  $m\in M$. Here the symbol ${\rm ord}_{y}(f)$ stands for the vanishing order of $f$ at the point $y$. 
Actually, to any $C\in H(\EE)$ we pick any valuation $\nu \in C$ living in its relative interior (i.e., in the complement of the union of the proper hyperfaces of $C$). The center in $X$ of the valuation $\nu$
is exactly an orbit closure $\bar{O}_{C}\subseteq E$ and does not depend on the choice of $\nu$. The correspondence is therefore given by $C\mapsto O_C$. 

\subsection{Preliminaries on intersection cohomology}\label{sec-ic-conv}
In this section, we fix the notation on the intersection cohomology theory that we will use throughout the paper.
Let $X$ be a complex algebraic variety and denote by $D^{b}_c(X)$ the constructible derived category of sheaves of $\QQ$-vector spaces on $X$. This is a triangulated category with shift functor $[1]$. For any object $\mathcal{F}\in D^{b}_c(X)$  we write $\mathcal{H}^{j}(\mathcal{F})$ for its cohomology sheaves. Given a morphism of
complex algebraic varieties $f: X\rightarrow Y$, we consider the functors 
$$f_{\star}: D^{b}_c(X)\rightarrow D^{b}_c(Y),\, f_{!}: D^{b}_c(X)\rightarrow D^{b}_c(Y), \text{ and }$$
$$f^{\star}: D^{b}_c(Y)\rightarrow D^{b}_c(X),\, f^{!}: D^{b}_c(Y)\rightarrow D^{b}_c(X).$$
Note that we consider only, in this paper, derived functors  and write $f_{\star}$ instead of $Rf_{\star}$, etc. We will also denote by
 $\mathbb{D}:D^{b}_c(X)\rightarrow D^{b}_c(X)$ the Verdier duality. 
\\

We start with a  Whitney stratification 
$$X = \bigsqcup_{\lambda \in I} X_{\lambda}.$$
Let $X_{\lambda_{0}}$ be the open stratum, let  $i_{\lambda}: X_{\lambda}\rightarrow X$ be the natural inclusion 
and let $\mathscr{L}$ be a local system defined on  $X_{\lambda_{0}}$.  We denote by  $IC_{X}(\mathscr{L})$ the \emph{intersection cohomology complex} with coefficients in $\mathscr{L}$. This constructible complex of sheaves is uniquely determined by the following  conditions. 
\begin{itemize}
\item[(1)] The open stratum condition: $i^{\star}_{\lambda_0} IC_{X}(\mathscr{L})= \mathscr{L}[\dim\, X]$.
\item[(2)] The stalk conditions: $\mathcal{H}^j(i^{\star}_{\lambda} IC_{X}(\mathscr{L}))= 0$ for $j\geq - \dim\, X_{\lambda}$ and $\lambda\neq \lambda_0$.
\item[(3)] The costalk conditions: $\mathcal{H}^j(i^{!}_{\lambda} IC_{X}(\mathscr{L}))= 0$ for $j\leq - \dim\, X_{\lambda}$ and $\lambda\neq \lambda_0$.
\end{itemize}
In our convention,  the intersection cohomology complex $ IC_{X}(\mathscr{L})$ belongs to the category of \emph{perverse sheaves} (for the middle perversity), that is, the heart of the category $D^{b}_c(X)$ with respect to the $\mathfrak{t}$-structure
$$ {}^pD^{\leq 0}(X):= \{ \mathcal{F}\in D^{b}_c(X)\,|\, \dim\, {\rm Supp}(\mathcal{H}^j(\mathcal{F}))\leq -j\text{ for all }j\} \text{ and }$$ 
$$ {}^pD^{\geq 0}(X):= \{ \mathcal{F}\in D^{b}_c(X)\,|\, \dim\, {\rm Supp}(\mathcal{H}^j(\mathbb{D}\mathcal{F}))\leq -j\text{ for all }j\}.$$ 
The \emph{intersection cohomology groups} with coefficients in $\mathscr{L}$ are defined by taking the hypercohomology groups: 
$$ IH^{i}(X; \mathscr{L}) := \mathbb{H}^{i}(X, IC_{X}(\mathscr{L})[-\dim\, X]).$$
In particular, if $\mathscr{L} = \QQ$, then we set $IC_X := IC_{X}(\mathscr{L})$ and 
$IH^{i}(X; \QQ) := \mathbb{H}^{i}(X, IC_{X}[-\dim\, X])$ is the intersection cohomology groups with rational coefficients. Observe that if $X$ is smooth, then $IC_{X} = \QQ[\dim\, X]$ and so $IH^{i}(X; \QQ) =  H^i(X;\QQ)$ is the usual cohomology groups with rational coefficients.
\\

The next result is the \emph{decomposition theorem} due to 
Beilinson, Bernstein, Deligne and Gabber which allows to describe the topology of algebraic proper maps. 
\begin{theorem}\label{t-bbdtheodec}\cite[Theorem 6.25]{BBD82}
Let $f: X\rightarrow Y$  be a proper algebraic morphism between two complex algebraic varieties $X$ and $Y$. Then there exists
a finite family $(Z_{\alpha}, \mathscr{L}_{\alpha}, d_{\alpha})$ where for every index $\alpha$, the subset $Z_{\alpha}\subseteq Y$ is a smooth Zariski locally closed subvariety, the symbol $\mathscr{L}_{\alpha}$ stands for a semi-simple local system on $Z_{\alpha}$ (that is, the representation of the fundamental group $\pi_{1}(Z_{\alpha}, y)$ on $Z_{\alpha}$ corresponding to $\mathscr{L}_{\alpha}$ is semi-simple) and $d_{\alpha}\in\ZZ$ such that we have an isomorphism
$$f_{\star} IC_{X}\simeq \bigoplus_{\alpha} (i_{\alpha})_{\star} IC_{\bar{Z}_{\alpha}}(\mathscr{L}_{\alpha})[-d_{\alpha}]$$
in the derived category $D^{b}_{c}(Y)$, where $\bar{Z}_{\alpha}$ is the Zariski closure of $Z_{\alpha}$ in $Y$ and $i_{\alpha}: \bar{Z}_{\alpha}\rightarrow Y$ is the natural inclusion. 
\end{theorem}

 Let $f: X\rightarrow Y$ be a surjective algebraic map
and consider a smooth stratification 
$$Y = \bigsqcup_{\lambda \in J} Y_{\lambda}$$
with connected strata such that the maps $f^{-1}(Y_{\lambda})\rightarrow Y_{\lambda}$ are topologically locally trivial fibrations for any $\lambda\in I$. We say that $f$ is \emph{semi-small} if for any $\lambda\in J$ and any $y\in Y_{\lambda}$ we have $$\dim\, f^{-1}(y)\leq \frac{1}{2}(\dim\, Y - \dim\, Y_{\lambda}).$$
In this case, we say that a stratum $Y_{\lambda}$ is \emph{relevant} if the equality
$$ \dim\, f^{-1}(y) =  \frac{1}{2}(\dim\, Y - \dim\, Y_{\lambda})  \text{ holds for some }y\in Y_{\lambda}.$$ 
The decomposition theorem has a more explicit form in the context of semi-small algebraic maps with rationally smooth total space. 
\begin{theorem}\cite[Proposition 8.2.21, Theorem 8.2.36]{HTT08}\label{theo-semi-decomp}
Let $f: X\rightarrow Y$ be a proper algebraic surjective semi-small morphism between two complex algebraic varieties $X$ and $Y$. Assume 
that $X$ is rationally smooth. Then we have $IC_{X}\simeq \QQ_{X}[\dim\, X]$. Furthermore, there exists
a canonical family $(Z_{\alpha}, \mathscr{L}_{\alpha})$ where for every index $\alpha$,  the subset $Z_{\alpha}\subseteq Y$ is a smooth Zariski locally closed subvariety, and $\mathscr{L}_{\alpha}$ is a semi-simple local system defined on  $Z_{\alpha}$ such that 
$$f_{\star} IC_{X} \simeq \bigoplus_{\alpha}  (i_{\alpha})_{\star} IC_{\bar{Z}_{\alpha}}(\mathscr{L}_{\alpha}).$$
 The data $(Z_{\alpha}, \mathscr{L}_{\alpha})$ satisfies the following conditions. 
\begin{itemize}
\item[(i)] The varieties $Z_{\alpha}$ are the relevant strata of a smooth stratification of $Y$ such that the maps  $$f^{-1}(Z_{\alpha})\rightarrow Z_{\alpha}$$ are topologically locally trivial fibrations. 
\item[(ii)] The local system $\mathscr{L}_{\alpha}$ is the local system associated with 
$$y\mapsto H^{ \dim\, Y - \dim\, Z_{\alpha}}(f^{-1}(y); \QQ).$$
\end{itemize}
\end{theorem}

\subsection{Seifert torus bundles}
Let $G$ be a connected linear algebraic group and let $H\subseteq G$ be an algebraic subgroup acting on  a variety $Z$. 
The group $H$ acts on the product by $h\cdot (g, z) = (gh^{-1}, h\cdot z)$, where $g\in G$, $h\in H$ and $z\in Z$. If there exists a 
geometric quotient $(G\times Z)/H$, then we set $G\times^{H}Z:= (G\times Z)/H$ and call it the \emph{homogeneous fiber space} over $G/H$ associated with $Z$. Note that the geometric quotient $(G\times Z)/H$ always exists if  $Z$ is quasi-projective (see \cite[Theorem 4.19]{PV89}). In this section, we introduce the concept of Seifert torus bundles that will be used after. They form a  particular class of homogeneous fiber spaces.
\begin{definition}\label{def-Seifertb}
By a \emph{Seifert torus bundle} we mean the data $(X, \TT, \Gamma)$,
where $\TT =  (\CC^{\star})^n$ is an algebraic torus, $\Gamma\subseteq \TT$ is a finite subgroup, and $X$ is a normal complex algebraic variety with an algebraic $\Gamma$-action admiting a covering by $\Gamma$-stable quasi-projective Zariski open subsets. In particular, there exists a geometric quotient $X\rightarrow X/\Gamma$. The \emph{total space} of the  Seifert torus bundle $(X, \TT, \Gamma)$ is the space  $V:= \TT\times^{\Gamma} X$, while the \emph{base} is the quotient $B:= X/\Gamma$. Finally, the \emph{fibration} (that we also call the
Seifert torus bundle) is the natural invariant projection $\varepsilon: V\rightarrow B$.
\end{definition}
In \cite{BM99}, Braden and MacPherson studied the \emph{Seifert line bundles}, which are roughly speaking fibrations that are locally the line bundles up to the action of a finite group. The following lemma justifies our terminology.
\begin{lemma}\label{l-quotquot}
Let $\varepsilon: V  =  \TT\times^{\Gamma} X\rightarrow B$ be a Seifert torus bundle. Consider the action of $\Gamma$ on the product $\TT\times X$ given by $g\cdot (x, y)  =  (g\cdot x, y)$, where $g\in \Gamma$, $x\in \TT$ and $y\in X$. This action naturally induces a $\Gamma$-action on $V$. Then we have the isomorphism $V/\Gamma\simeq \TT/\Gamma\times B$. 
\end{lemma} 
\begin{proof}
We have two natural maps $\varepsilon: V\rightarrow B$ and $V\rightarrow \TT/\Gamma$ inducing a $\Gamma$-invariant morphism $V\rightarrow \TT/\Gamma\times B$. By the universal property of the quotient, we get a morphism $f: V/\Gamma\rightarrow B\times \TT/\Gamma$. By construction $f$ is bijective. Since $V/\Gamma$ is a normal variety, we conclude by Zariski Main Theorem that $f$ is an isomorphism.
\end{proof}
\section{Stratifications and fibers}\label{s-one}
\subsection{Fibers of the contraction maps}
This section is dedicated to studying the geometry of the contraction map. We will use notations of the introduction. 
 Namely $\pi: \tilde{X}\rightarrow X$ is the contraction map of our normal complexity-one $\TT$-variety $X$, the set $E$ is the image of the exceptional locus of $\pi$ and the morphism
$\gamma: \tilde{X}\rightarrow Y_{0}$ is the surjective global quotient. Note that we assume that the $\TT$-action on $X$ is faithful.  Finally, for any point $y\in Y_{0}$ we  write $F_{y}= \gamma^{-1}(y)_{\rm red}$ for its reduced fiber. 
\\

We begin by observing that
the $\TT$-variety $\tilde{X}$ has a $\TT$-stable open subset of the form $Y'\times V$,
where $V$ is the toric variety corresponding to the general fiber of the quotient map $\gamma$ and $Y'\subseteq Y_{0}$ is a Zariski dense open subset. On this subset, the torus trivially acts on the first factor and by the natural action on the second factor. Elements of the set $Z:= Y_{0}\setminus Y'$ will be called \emph{special}. The following lemma describes the preimage of any orbit $O\subseteq E$ under $\pi$.
\begin{lemma}\label{l-fiber}
Let  $O$ be a $\TT$-orbit in $X$ contained in $E$. 
Denote by $\TT_{x}$ the isotropy group of a point $x\in O$. Write $\TT^{0}_{x}$ for the neutral connected component of $\TT_{x}$ and $\Gamma_{x}$ for the finite group $\TT_{x}/\TT^{0}_{x}$. Set $\TT_{O}:= \TT/\TT^{0}_{x}$. Then the fiber $\pi^{-1}(x)$ is a smooth projective curve and we have the homogeneous fiber space decomposition
$$ \pi^{-1}(O)\simeq \TT_{O} \times^{\Gamma_{x}}\pi^{-1}(x).$$
\end{lemma}

Before proving Lemma \ref{l-fiber}, we make the following remark.

\begin{remark}\label{r-fiberprod}
Let $G$ be a connected linear algebraic group acting on a quasi-projective variety $S$.
If $f: S\rightarrow G/H$ is a $G$-equivariant morphism onto a homogeneous $G$-space, then $S$ is $G$-isomorphic to the homogeneous fiber space $G\times^{H}F$, where $F = f^{-1}(H/H)$ (see \cite{Ser58} or the comment after \cite[Theorem 2.2]{Tim11}). Note that $f$ is a locally trivial  fibration for the \'etale topology.
\end{remark}

\begin{proof}
Consider the $\TT$-equivariant morphism  $\pi^{-1}(O)\rightarrow O = \TT/\TT_{x}$. Then
by Remark \ref{r-fiberprod} we have a  homogeneous fiber space decomposition $ \pi^{-1}(O)\simeq \TT\times^{\TT_{x}}\pi^{-1}(x).$ Moreover,
observe that $\pi^{-1}(O)$ intersects the Zariski open subset $Y'\times V$ and that the fibers of $\pi$ are connected according to Zariski Main Theorem.
It follows that $ \pi^{-1}(O)\cap (Y'\times V)$ is a connected $\TT$-stable  Zariski closed subset of $Y'\times V$. This implies that 
$ \pi^{-1}(O)$ is irreducible.
 Hence the locally closed subvariety $ \pi^{-1}(O)$ must be  normal (compare \cite[Theorem 16.25]{Tim11}). Note that the fibration 
$$\pi_{|\pi^{-1}(O)}: \pi^{-1}(O)\simeq \TT\times^{\TT_{x}}\pi^{-1}(x)\rightarrow \TT/\TT_{x}\simeq O$$ is locally trivial for the \'etale topology.  So  $\pi^{-1}(x)$ is normal and irreducible. Since the orbits of $\pi^{-1}(O)$ are
parameterized by a curve, one concludes that $\pi^{-1}(x)$ is a smooth projective curve. In particular, the group $\TT^{0}_{x}$ trivially acts on $\pi^{-1}(x)$. Indeed,  the quotient of $\pi^{-1}(x)$ by $\TT^{0}_{x}$ is one-dimensional and therefore the group $\TT^{0}_{x}$ acts on  $\pi^{-1}(x)$ by finite automorphisms. Now the connectedness of $\TT^{0}_{x}$ implies that  the  $\TT^{0}_{x}$-action on $\pi^{-1}(x)$ is trivial. This gives the decomposition 
$ \pi^{-1}(O)\simeq \TT_{O} \times^{\Gamma_{x}}\pi^{-1}(x)$ and ends the proof of the lemma.
\end{proof}
\begin{remark}\label{rem-mistake}
In a former version of the present article, we claimed that  $\pi^{-1}(O)$ is equivariantly isomorphic to the trivial product $Y\times O$ whenever $O\subseteq X$ is a $\TT$-orbit of $E$. This is not true in general. Our argument was based on   \cite[Theorem 10.1 (ii)]{AH06} which seems incorrectly stated as pointed out by the referee. One reason is that the isotropy groups on the open subset $Y'\times V$ must be connected (remember that a toric variety with faithful torus action has connected isotropy groups). Now if we have Zariski triviality, then $O$ must have connected isotropy groups as well (since $\pi^{-1}(O)$ intersects  $Y'\times V$). The following example (see \ref{conterex1}) due to the referee shows that $O$ can have disconnected isotropy groups. We thank the referee  for mentioning this inaccuracy. 
\end{remark}
\begin{example}\label{conterex1}
We work with the lattice $N = \ZZ^{2}$ and the strictly convex polyhedral cone $\sigma  =  \QQ_{\geq 0}\times \{0\}$. Consider the polyhedral divisor $\D =  \sum_{y\in \P^{1}_{\CC}}\D_{y}\cdot [y]$ over the projective line with non-trivial coefficients $$\D_{0} =  \left(1, \frac{1}{2}\right) + \sigma\text{ and }  \D_{\infty} =  \left(0, -\frac{1}{2}\right) + \sigma.$$ Let $t\in \CC(\P^{1}_{\CC})$ be a rational function 
satisfying ${\rm div}(t)  = [0] - [\infty]$ and set 
$$ x_1 =  \frac{1}{t}\chi^{(1,0)}, \, x_2 =  \frac{1}{t}\chi^{(1,1)}, \,x_3 =  \frac{1}{t}\chi^{(1,2)}, \,  x_4 =  \frac{1}{t}\chi^{(0,2)}, \, x_5 =  t\chi^{(0,-2)}.$$ 
Then we have $A( \P^{1}_{\CC}, \D) =  \CC[x_1, \ldots, x_5]$ (this can be checked via \cite[Theorem 2.4]{Lan13}) and, in particular, the corresponding $(\CC^{\star})^{2}$-variety $X = X(\D)$ is equivariantly isomorphic to 
$$\{(x_1, \ldots, x_5)\in \CC^5\,|\, x_1 x_3 - x_{2}^2 = x_4 x_5 - 1 = 0\}$$
with torus action given by 
$$(\lambda, \mu)\cdot (x_1, \ldots, x_5) = (\lambda x_1, \lambda \mu x_2, \lambda \mu^2 x_3, \mu^2 x_4, \mu^{-2}x_5)
\text{ for } (\lambda, \mu)\in  (\CC^{\star})^{2}.$$
Note that $t =  \frac{x_3}{x_1 x_4}$ and the rational quotient is given  by  $(x_1, \ldots, x_5)\mapsto (x_3: x_1 x_4)$.
Now the orbit $E = \{ (0,0, 0, x_4, x_5)\in \CC^5\, |,\,  x_4 x_5 - 1 = 0\}$ is the image of the exceptional locus of the contraction map
and has isotropy groups equal to $\CC^{\star}\times \mu_{2}(\CC)$.

From \cite[Section 3, Lemma 1]{Vol10}, one observes that the contraction map $\pi$, in this example, is obtained by blowing-up $E$.
One has $\pi^{-1}(E) \simeq \P_{\CC}^{1}\times \CC^{\star}$ as an abstract algebraic variety and the action on the coordinates is
given by 
$$ (\lambda, \mu)\cdot ([u:v], x_{4}) = ([u:\mu v]: \mu^{2}x_{4}) \text{ for any } (\lambda, \mu)\in (\CC^{\star})^{2}.$$
In this way, we see that the general orbits of $\pi^{-1}(E)$ are mapped onto $E$ via a double covering. 
\end{example}
We will use the following lemma later.
\begin{lemma}\label{l-compl}
The image $E$ of the exceptional locus of the contraction map $\pi:\tilde{X}\rightarrow X$ has finitely many $\TT$-orbits. 
\end{lemma}
\begin{proof}
This follows from the fact that  the torus action is of complexity one (we recall that the complexity of a torus action is the codimension
of the general orbits).
 Indeed, let $E_{1}$ be an irreducible component
of $E$ (which is automatically stable by the torus action). Then
by \cite[Theorem 5.7]{Tim11}, we know that the complexity of the torus action on $E_{1}$ is $0$ or $1$. Let us assume that the complexity is $1$. Using Lemma \ref{l-fiber},
the preimage $E_{2}$ of $E_{1}$ under the contraction map has complexity $2$, a contradiction with \emph{loc. cit.} Thus, the complexity of the torus action on $E_{1}$ is $0$ and the finiteness of the set of orbits of $E$ follows.
\end{proof}

\subsection{Constructing a stratification}
Next, we construct a natural stratification associated with the morphism $\pi:\tilde{X}\rightarrow X$. 
If $n$ is the complex dimension of the torus, then for $0\leq d\leq n+1$ we
write $\CS_{d}(\tilde{X})$ for the set of closed subsets $S$ in $\tilde{X}$ satisfying the 
following conditions. The set $S$ is  a Zariski closure of  a set of the form $Y'\times O$, where $O$ is a $(d-1)$-dimensional orbit of the toric variety $V$ (this class of sets is not considered if $d =0$), or a $d$-dimensional orbit closure contained in a special fiber. 
\begin{lemma}\label{l-stratum}
The sequence $X_{\bullet}$ of closed subsets
$$\emptyset = X_{-1}\subseteq X_{0}\subseteq X_{1}\subseteq \ldots \subseteq X_{n+1} = \tilde{X}, \text{ where }X_{d} = \bigcup_{S\in\CS_{d}(\tilde{X})}S,$$
defines a stratification that makes  $\tilde{X}$ a pseudo-manifold in the sense of \cite[Section 1.1]{GM83}. 
In addition, the image $\pi_{\star}X_{\bullet}$ of the filtration $X_{\bullet}$ by $\pi$ leads to a stratification on $X$, and $\pi$ is a topological fibration with respect to these strata. 
\end{lemma}
\begin{proof}
Let us prove that the filtration $X_{\bullet}$ gives rise to a stratification.
It is clear that $X_{n+1}\setminus X_{n}$ is dense in $\tilde{X}$. Furthermore, for every $d$ the subset $X_{d} \setminus X_{d-1}$ is the disjoint union of a finite set of orbits of dimension $d$ and sets of the form $O \times Y'$, where $O$ is an orbit of dimension $d-1$. So $X_{d} \setminus X_{d-1}$ is smooth.

 Assume that $X_{d} \setminus X_{d-1}\neq \emptyset$ for $0\leq d\leq n+1$. Let us fix an $x\in X_{d} \setminus X_{d-1}$. Then there are two possibilities.
\begin{enumerate}
  \item The point $x$ belongs to an orbit of dimension $d-1$, contained in the fiber $F_{y}$ for some $y\in Y'$. Then consider an open neighborhood of $x$ in $\tilde{X}$ as a product $V_x=V_{y,Y} \times V_{x,F_{y}}$, where $V_{y,Y}$ is an open neighborhood of $y$ in $Y$ such that $V_{y,Y}\cap Z=\emptyset$ and $V_{x,F_{y}}$ is an open neighborhood of $x$ in $F_{y}$. Reducing $V_{y,Y}$ and $V_{x,F_{y}}$ if necessary, we may suppose that $V_{y,Y} \simeq \mathbb{R}^{2}$ and $V_{x, F_{y}} \simeq \mathbb{R}^{2(d-1)}\times c^0(L)$ (as pseudo-manifolds), where $c^0(L)$ is the cone of a stratified space $L$ with real dimension $2(n-(d-1))-1$. The last isomorphism is induced from the natural stratification of the toric variety $F_{y}\simeq V$ in which the strata consist of orbits of the same dimension.
  \item The point $x$ belongs to an orbit of dimension $d$, contained in $F_{y}$ for some $y\in Z$. Using that $\tilde{X}$ is toroidal, we may embed $F_{y}$ into a Zariski
  open subset of $\tilde{X}$, which is \'etale to an open set of an affine toric variety. It follows from the toric structure that there is an open neighborhood $W$ of $x$ which is isomorphic to $\mathbb{R}^{2d}\times c^0(L)$, where $L$ is a stratified space $L$ of real dimension $2(n +1 -d)-1$. 
\end{enumerate}
This shows that $X_{\bullet}$ defines a stratification. 

The map $\pi$ is a topological fibration since its restriction on each connected component of $X_{d}\setminus X_{d-1}$
is an \'etale morphism or a fibration with fiber $Y_{0}$ which is locally trivial for the \'etale topology (see Lemma \ref{l-fiber}).
Finally, from this one deduces that $\pi_{\star}X_{\bullet}$ also provides
a stratification, proving our lemma.
\end{proof}
\begin{lemma}\label{l-semi}
The contraction map $\pi:\tilde{X}\rightarrow X$ is semi-small, i.e., the inequality $\dim\, X - \dim \, \pi_{\star}X_{d}\geq 2\,\dim\, \pi^{-1}(x)$ holds for every integer $d$ such that $0\leq d\leq n+1$ and any point $x\in \pi(X_{d})$ lying in the stratum.  
\end{lemma}
\begin{proof}
Let $x\in X$. By Lemma \ref{l-fiber}, the real dimension of $\pi^{-1}(x)$ is either $0$ or $2$. Suppose that the real dimension of $\pi^{-1}(x)$ is $2$ and let $O_x$ be the orbit of $x$. Then $O_x$ is a connected component of a stratum of $\pi_{\star}X_{\bullet}$. Given $y\in Y_{0}$, let $O_y\subseteq \tilde{X}$ be the orbit such that $\gamma(O_y)=\{y\}$ and $\pi(O_y)= O_x$. Hence the complex codimension of $O_x$ is $\codim_X O_x =1 + \codim_{F_{y}} O_y \geq 2$ which proves that $\pi$ is semi-small.
\end{proof}
\begin{remark}\label{r-str}
Recall that a relevant stratum of $\pi_{\star}X_{\bullet}$ is a connected component $S_{d}$ of a stratum $\pi_{\star}X_{d}\setminus \pi_{\star}X_{d-1}$ such that the equality
$\frac{1}{2}(\dim\, X - \dim\, S_{d}) = \dim\, \pi^{-1}(x)$ holds for some $x\in \pi(S_{d})$. Actually, such a stratum is 
either the big open stratum $\pi_{\star}X_{n+1}\setminus \pi_{\star}X_{n}$, or the codimension-two torus orbits in $E\subseteq X$.    
\end{remark}

\section{Decomposition theorem}\label{s-two}

In this section, we focus on the proof of Theorem \ref{t-main}. 
We keep the same notations as in Section \ref{s-one}. Namely $\pi: \tilde{X}\rightarrow X$ is the contraction map of our $\TT$-variety $X$, the morphism
$\gamma: \tilde{X}\rightarrow Y_{0}$ is the surjective global quotient map, and $V$ is 
its general fiber. We also assume that $Y_{0}$ is a complete curve.
\\

\subsection{Local slices}
Our first aim is to describe the local structure of the contraction map (see \ref{subsec-cont-map-loc}). We first start with some preparations.
\begin{lemma}\label{l-redu}
Let $G$ be a  reductive connected linear algebraic group acting on an affine variety $B$ and assume that the algebra of  invariants $\CC[B]^G$ is equal to $\CC$. Then $B$ has exactly one closed orbit, and if $H$ is the isotropy group of any point of this orbit, then
there exists an action of $H$ on an affine reduced  scheme $B_0$ of finite type over $\CC$ such that the following 
assertions hold.
\begin{itemize}
\item[(1)] $B_0$ contains exactly one closed orbit, which is a fixed point, and any orbit closure in $B_0$ contains this point. In
particular, $B_0$ is connected. 
\item[(2)] We have a  homogeneous fiber space decomposition $B\simeq G\times^H B_0$. In particular,  if $B$ is normal, then $B_0$ is normal
and  irreducible. 
\end{itemize}
\end{lemma}
\begin{proof}
The proof can be found in \cite[Section 6.7]{PV89}.
\end{proof}
Now we apply the preceding result in the special case of torus actions of complexity one.
\begin{lemma}\label{l-pointf} Let $C$ be a smooth projective curve and let $\D\in \PPDiv(C, \sigma)$. Denote by $R =  A(C, \D)$ the associated algebra and by $B =  \spec\, R$ the associated $\TT$-variety. Consider the following sublattices
$$M_1 =  \{m\in M\, |\, \{m, -m\}\subseteq \sigma^\vee\}\text{ and }M' =  \{ m\in M_1\, |\, R_m\neq \{0\}\}, $$
where $R_m$ is the graded piece of $R$ of degree $m$. Also set $N_0 = \{v\in N\, |\, \langle  m , v\rangle = 0 \text{ for any }m\in M_1\}$. Pick a point $x$ in the unique closed orbit $O$ of $B$. Then we have the isomorphisms
$$ O\simeq \TT/\TT_{x}\simeq {\rm Hom}( M', \CC^{\star})\text{ and } \TT/\TT_{x}^0\simeq {\rm Hom}( M_1, \CC^{\star}).$$
Moreover, since $N_0$ is satured in $N$, we may consider a direct sum decomposition $N  =  N_0\oplus N_1$ giving rises
to a torus decomposition $\TT =  \TT_{N_0}\times \TT_{N_1}$. Then, with the notations of Lemma \ref{l-fiber}, we have an identification $\TT_O \simeq  \TT_{N_1}$
and a  homogeneous fiber space decomposition $B\simeq \TT_O\times^{\Gamma_x}B_1$, where $B_1$ is a normal
affine $\TT_O$-variety.
\end{lemma}
\begin{proof}
The vector space generated by $M_1$ in $M_{\QQ}$ is a face of  $\sigma^\vee$. From this, one sees that 
$I =  \bigoplus_{m\in (\sigma^\vee\cap M)\setminus M_1}R_m$ is an ideal of $R$ whose quotient 
is isomorphic to $R' =   \bigoplus_{m\in M'}R_m$. Using the properness conditions on $\D$, one observes
that each piece $R_m$ for $m\in M'$ is one-dimensional and therefore $R'$ is the algebra of regular functions
of $O$. The quotient map
$$\varepsilon: B\simeq \TT\times^{\TT_x}B_1\rightarrow \TT/\TT_x\simeq O$$
from Lemma \ref{l-redu} is actually given by the inclusion of rings $R'\subseteq R$. But also note that $R' =  R^{\TT_{N_0}}$ by the very definition of $N_0$, and so $\varepsilon$ is the global quotient for the  $\TT_{N_0}$-action on $B$. Remarking that $\varepsilon$ is  $\TT_{N_1}$-equivariant and that  $\TT_{N_1}$ transitively acts on $O$ with
stabilizers isomorphic to $\Gamma_x$, one obtains that $\varepsilon$ induces the  homogeneous fiber space decomposition  $B\simeq \TT_O\times^{\Gamma_x}B_1$, as required.
\end{proof}
\begin{corollary}\label{cor-self}
With the same notation of Lemma \ref{l-pointf}, the point $x\in O$ has connected stabilizers if and only if  for every $m\in \sigma^{\vee}\cap M$ such that $\min_{v\in \deg(\D)}\langle m , v\rangle =  0$, the divisor $\D(m)$ is principal.
\end{corollary}
\begin{proof} Self evident. 
\end{proof}

\subsection{Local structure of the contraction map}\label{subsec-cont-map-loc}
We now deal with the initial notations from the beginning of Section \ref{s-one}. 
In the sequel we will fix an orbit $O\subseteq E$ of the image of the exceptional locus of the contraction map $\pi$. 
For any $y\in Y_{0}$ we denote by $O_{y}$ the $\TT$-orbit in $\pi^{-1}(O)$ mapping onto $y$ via the quotient map. Let us consider the subset
$$\tilde{X}_{O}:= \{z\in \tilde{X}\,|\, O_{\gamma(z)}\subseteq \overline{\TT\cdot z}\}.$$
Our local structure result can be expressed as follows. 
\begin{lemma}\label{l-structure}
The subset $\tilde{X}_{O}$ is a Zariski dense open subset of $\tilde{X}$, and an orbit
$O'\subseteq \tilde{X}_{O}$ is closed in $\tilde{X}_{O}$ if and only if $O'\subseteq \pi^{-1}(O)$. There exists a
 homogeneous fiber space decomposition  $\tilde{X}_{O}\simeq \TT_O\times^{\Gamma_x}\tilde{X}_1$  which gives rise to
a Cartesian commutative diagram
 $$\xymatrix{
      \tilde{X}_O \simeq \TT_O\times^{\Gamma_x}\tilde{X}_1\,\,\ar[r] \ar[d] ^{\pi} & \tilde{X}_{1}/\Gamma_x \ar[d]^{\pi_{1}}\\X_O \simeq \TT_O\times^{\Gamma_x}X_1\,\, \ar[r] & X_{1}/\Gamma_x}.$$ 
Here the horizontal arrows are Seifert torus bundles (see Definition \ref{def-Seifertb}) and the vertical arrows are proper morphisms. The image $X_O = \pi(\tilde{X}_O)$ is affine Zariski open in $X$ and has $O$ as unique closed orbit. Moreover, if we denote again by $x$ the image of $x$ in $X_{1}/\Gamma_{x}$, then the preimage of $x$ under $X_{O}\rightarrow X_{1}/\Gamma_{x}$ is $O$. 
\end{lemma} 
\begin{proof}
From the definition of $\tilde{X}_{O}$, we remark that for general $y\in Y_{0}$ the set $F_{y}\cap \tilde{X}_{O}$ is an affine $\TT$-stable Zariski open subset of the toric variety $V$ that does not depend
on $y$. Let $\sigma$ be the corresponding cone in the fan of $V$. From the description of $\tilde{X}$ in terms of divisorial fans in \cite[Theorems 5.3, 5.6]{AHS08}, we may find a $\TT$-stable Zariski open subset $X_{0}$
of $\tilde{X}$ which is the relative spectrum of a $\sigma$-polyhedral divisor $\D$ with locus $Y_{0}$ and containing $\pi^{-1}(O)$. Strictly speaking, we choose $\D$ to be exactly the polyhedral divisor corresponding to the hyperface associated with the orbit closure $\bar{O}\subseteq X$ in the sense of \cite[Theorem 16.19]{Tim11} (see also the discussion at the end of Section \ref{sec-r-tvar}). Using the dictionary between orbit closures and face relations from \emph{loc. cit.} and the definition of $\tilde{X}_{O}$, we get $\tilde{X}_{O}= X_{0}$.

We prove the existence of the decomposition  $\tilde{X}_{O}\simeq \TT_O\times^{\Gamma_x}\tilde{X}_1$.
By the description of $\tilde{X}_{O}$ in terms of hyperfaces, one sees that the image $X_O = \pi(\tilde{X}_{O})$
is open, affine  and $\TT$-stable in $X$ and has $O$ as unique closed orbit. Now applying Lemma  \ref{l-pointf} we get a  homogeneous fiber space decomposition $X_O \simeq \TT_O\times^{\Gamma_x}X_1$ and a $\TT_O$-equivariant map $\varepsilon: X_O\rightarrow \TT_O/\Gamma_x$ (note that we choose a decomposition $\TT\simeq \TT_{N_0}\times \TT_{N_1}$ as in  \ref{l-pointf}
and we use the identification  $\TT_O \simeq  \TT_{N_1}$). Hence by composing we have a $\TT_O$-equivariant map $\tilde{X}_O\rightarrow \TT_O/\Gamma_x$ yielding the desired decomposition $\tilde{X}_{O}\simeq \TT_O\times^{\Gamma_x}\tilde{X}_1$ (see Remark \ref{r-fiberprod}).
By construction, it follows that  $X_1 = \varepsilon^{-1}(x)$ and $\tilde{X}_1 = \pi^{-1}(\varepsilon^{-1}(x))$.
Hence the morphism $\pi_1: \tilde{X}_1/\Gamma_x \rightarrow X_1/\Gamma_x$ is induced by restriction of the contration map $\pi$.
It is then clear that we have a Cartesian commutative diagram as in the statement.  
\end{proof} 

\begin{lemma}\label{l-loc} 
Consider the decomposition  given by
Lemma \ref{l-structure}, where $O$ is an orbit in the image $E$ of the exceptional locus of $\pi$.
Assume that the stabilizers of $O$ are connected and  set $r = \dim\, O$. 
Then we have the isomorphisms
$$\H^{j}(\pi_{\star}IC_{\tilde{X}})_{|O} \simeq  \H^{j+r} ((\pi_{1})_{\star}IC_{\tilde{X}_{1}})_{x}\otimes \QQ_{O}.$$
 In particular, each restriction of the sheaf $\H^j(\pi_{\star} IC_{\tilde{X}})$ on every orbit is constant. 
\end{lemma}
\begin{proof}
The idea of the proof is inspired by the one of Bernstein and Lunts in \cite[Lemma 5.15]{BL94}.
By Lemma \ref{l-structure}, we have the Cartesian commutative diagram 
 $$\xymatrix{
      \tilde{X}_O \simeq \TT_O\times\tilde{X}_1\,\,\ar[r]^{\,\,\,\,\,\,\,\,\,\,\,\,\,\,\,p} \ar[d] ^{\pi} & \tilde{X}_{1}\ar[d]^{\pi_{1}}\\X_O \simeq \TT_O\times X_1\,\, \ar[r]^{\,\,\,\,\,\,\,\,\,\,\,\,\,\,\,q} & X_{1},}$$     
where the horizontal arrows $p$ and $q$ are the natural projection the second factor.
Indeed, the connectedness of stabilizers of the points of $O$ implies that the group $\Gamma_{x}$ is trivial in the description of \emph{loc. cit.}
  Now, since $p$ is a smooth morphism with fibers of dimension $r$,  we have that $IC_{\tilde{X}} \simeq  p^{\star}IC_{\tilde{X}_{1}}[r]$. This follows from the results in \cite[Section 4.2]{BBD82} by
adapting \cite[Lemma 2.4]{AL17} for smooth morphisms.
 Using proper base change (see \cite[Theorem 2.3.26]{Dim04}), it follows that 
$$\H^{j}(\pi_{\star}IC_{\tilde{X}})_{|O}\simeq \H^{j}(\pi_{\star}p^{\star}IC_{\tilde{X}_{1}}[r])_{|O}\simeq\H^{j}(  i^{\star}q^{\star}(\pi_{1})_{\star}IC_{\tilde{X}_{1}}[r]),$$
where $i: O\rightarrow X$ is the inclusion. Finally, we remark that the composition $q\circ i$ is the projection on $O$ onto the point $x$ of $X_{1}$. Therefore 
$$\H^{j}(  i^{\star}q^{\star}(\pi_{1})_{\star}IC_{\tilde{X}_{1}}[r])\simeq \H^{j+r} ((\pi_{1})_{\star}IC_{\tilde{X}_{1}})_{x}\otimes \QQ_{O},$$
as expected.
\end{proof}

\begin{proof}[Proof of Theorem \ref{t-main}]
$(i)$ Geometrically, one can take for each orbit $O\subseteq E$ and each point $x\in E$ the quotient by the group $\Gamma_{x}\subseteq \TT$ in the commutative diagram of Lemma \ref{l-structure}. This has the effect to change, in this diagram, 
the Seifert torus bundles  by Zariski trivial torsors with fibers isomorphic to $\TT_{O}/\Gamma_{x}$ (see Lemma \ref{l-quotquot}). Thus the image of $O$ in the resulting new $\TT/\Gamma_{x}$-variety $X/\Gamma_{x}$ have trivial stabilizers. Repeating this process for each orbit $O$, we get the desired commutative diagram of Theorem \ref{t-main} $(i)$. 

Algebraically, one can alternatively construct the group $G$ of \ref{t-main} $(i)$ by taking a finite index sublattice $M_{0}$ of the lattice $M$. The lattice $M_{0}$ has the 
property that for any $\sigma$-polyhedral divisor $\D$ having complete locus in the divisorial fan $\EE$ defining $X$ and any $m\in M_{0}$, we have $\D(m)$ principal whenever we have $\min_{v\in \D_{y}}\langle m, v\rangle = 0$ (see Corollary \ref{cor-self}). Then we set $G = {\rm Hom}(M/M_{0}, \CC^{\star})$ and consider the natural $G$-actions on $X$ and $\tilde{X}$.

Now for the comparison with the Betti numbers, we have, 
 according to Kirwan  \cite[Lemma 2.12]{Kir86}, the isomorphisms
$$IH^{j}(X; \QQ)^{G} \simeq IH^{j} (X/G; \QQ) \text{ and } IH^{j}(\tilde{X}; \QQ)^{G} \simeq IH^{j} (\tilde{X}/G; \QQ),$$
see also \cite[Proposition 5.1]{GH17} for a detailed proof communicated by MacPherson.
Moreover, since the action of $G$ extends to the action of the connected Lie group $\TT$, the actions of $G$ on the vector spaces $IH^{j}(X; \QQ)$ and $IH^{j}(\tilde{X}; \QQ)$ are trivial. Indeed,
as explained in \cite[Lemma 2.12]{Kir86}, one may reduce to the case where $X$ and $\tilde{X}$ are manifolds by taking equivariant resolutions of singularities and applying the decomposition theorem (see Theorem \ref{t-bbdtheodec}) for these maps. Doing this
reduction, the triviality of the actions of $\TT$ (and therefore of $G$) on $H^{j}(X; \QQ) = IH^{j}(X; \QQ)$ and $H^{j}(\tilde{X}; \QQ) = IH^{j}(\tilde{X}; \QQ)$ follows by adapting \cite[Proposition 6.4]{DL76} for cohomology with rational coefficients. This ends the proof of Part $(i)$.
\\

$(ii)$  We assume that the stabilizers of the points of $E$ are connected. Also we may and do assume that the global quotient $Y_{0}$ is a complete curve (otherwise the contraction map is an isomorphism and the theorem follows).
The decomposition theorem (see Theorem \ref{t-bbdtheodec}) yields an isomorphism
$$ \pi_{\star}IC_{\tilde{X}} \simeq IC_X \oplus \bigoplus_{\alpha\in I} (\iota_{\alpha})_{\star}IC_{\bar{V}_\alpha}(\L_\alpha)[-d_\alpha],$$
where $I$ is a finite set, $V_\alpha$ are irreducible smooth Zariski locally closed subvarieties, the map $\iota_{\alpha}:\bar{V}_{\alpha}\rightarrow X$ is the natural inclusion,
$d_\alpha \in \ZZ$, and $\L_\alpha$ is a semi-simple local system on $V_\alpha$. As $\pi_{|\pi^{-1}(X \setminus E)}$ is an isomorphism, it follows that 
$(\pi_{\star} IC_{\tilde{X}})_{|X \setminus E} \simeq IC_{X \setminus E}$ and  $V_{\alpha}\subseteq E$ for any $\alpha$. This implies that $\bar{V}_\alpha= \cup_{O \in \mathcal{O}(E)} \bar{O} \cap \bar{V_\alpha}$.

Note that $E$ has finitely many orbits (see Lemma \ref{l-compl}).
This last remark implies that $\bar{V}_\alpha \subseteq \bar{O}$ for some orbit $O$ in $E$. Let $O_\alpha$ be a orbit of minimal dimension in $E$ such that $\bar{V}_\alpha \subseteq \bar{O}_\alpha$. Then $O_{\alpha}\cap V_{\alpha}\neq \emptyset$. Indeed, if not, the subvariety $V_{\alpha}$ would be contained in the orbit closure of the boundary of $O_{\alpha}$ and
this would contradict the minimality of $O_{\alpha}$.
Moreover, applying Lemma \ref{l-loc}, the sheaf $\H^{-\dim\,V_\alpha}((\iota_{\alpha})_{\star}IC_{\bar{V}_\alpha}(\L_\alpha))$ is constant on $O_\alpha$ and the restriction $\H^{-\dim\,V_\alpha}((\iota_{\alpha})_{\star}IC_{\bar{V}_\alpha}(\L_\alpha))_{|V_\alpha} \simeq \L_\alpha$ is locally constant. So the sheaf $\H^{-\dim \,V_\alpha}((\iota_{\alpha})_{\star}IC_{\bar{V}_\alpha}(\L_\alpha))$ is locally constant on $V_\alpha \cup O_\alpha$. This implies that $$O_\alpha \subseteq \supp (\H^{-\dim \,V_\alpha}((\iota_{\alpha})_{\star}IC_{\bar{V}_\alpha}(\L_\alpha))) \subseteq \bar{V}_\alpha.$$
Consequently $\bar{O}_\alpha=\bar{V}_\alpha$. This shows Assertion $(ii)$ of Theorem \ref{t-main}.
\\

$(iii)$ For the last part of Theorem \ref{t-main}, one uses the decomposition theorem for semi-small maps (see Theorem  \ref{theo-semi-decomp}). So we automatically assume 
that $\tilde{X}$ is rationally smooth. Also remember that $\pi$ is a semi-small map (see Lemma \ref{l-semi}). 
The fibers of $\pi$ are irreducible according to Lemma \ref{l-fiber} and therefore the local systems involving in the decomposition theorem are all trivial (see Condition $(ii)$ of  \ref{theo-semi-decomp}). Finally the analysis of the supports
comes from the study of the relevant strata in \ref{r-str}. This finishes the proof of the theorem. 
\end{proof}
\begin{example}\emph{ Complete $\CC^{\star}$-surfaces} (see also \cite{FK86}).\label{ex-surf}
Assume that $\dim\,X = 2$.
Then $\tilde{X}$ is rationally smooth and $E$ is a finite set with at most two elements. From Theorem \ref{t-main}, it follows that  
$$ \dim_{\QQ} \, IH^i(X; \QQ)= \dim_{\QQ} \, H^i(\tilde{X};\QQ)\text{ for }i = 0,1,3,4\text{ and } \dim_{\QQ}\, IH^2(X;\QQ)=\dim_{\QQ} \, H^2(\tilde{X}; \QQ) - |E|.$$ 
Note that the comparison between the intersection cohomology Betti numbers of $X$ and $\tilde{X}$ can be obtained from 
the comparison between the formulae in  Proposition 3.1 and Corollary 3.4 of \cite{FK86}, where the set of elliptic fixed points defined in \emph{loc. cit.} is precisely the set $E$. 
\end{example}

\section{Betti numbers}\label{s-three}
This section aims to establish Theorem \ref{t-hvector} which describes the intersection cohomology Betti numbers 
of any complete normal $\TT$-variety of complexity one. In order to do this, we will prove diverse results  on  intersection cohomology with torus action (see Sections \ref{sec-isol}, \ref{sec-attob}, \ref{sec-g-inv}).  
Recall that for every algebraic variety 
$Z$ we write $$P_{Z}(t) =\sum_{i =0}^{2d}b_{i}(Z)t^{i},\text{ where } d =  \dim\, Z\text{ and } b_{i}(Z)  =  \dim\, IH^{i}(Z; \QQ),$$
and call the polynomial $P_{Z}(t)$ the  \emph{Poincar\'e polynomial} of $Z$. Also note that, for the remainder of the paper, we will go on to use the notation on $\TT$-varieties of Section \ref{sec-r-tvar}.   We start with some preliminaries on the combinatorics and the topology of toric varieties. 

\subsection{Fans and intersection cohomology}\label{sec-sec-prem-prem}
It is known
(see for instance \cite{Sta87, DL91, Fie91}) that the intersection cohomology of an $n$-dimensional complete toric $\TT$-variety $\X$ with fan $\Sigma = \Sigma_{\X}$  (defined on $N_\QQ$) can be expressed
in terms of \emph{$h$-vectors}, namely that $P_{\X}(t)$ is determined by two relations
$$P_{\X}(t) = \sum_{\sigma \in \Sigma}(t^2-1)^{n-\dim\,\sigma}P_{\sigma}(t), \text{ and } P_\sigma(t) = \left\{
    \begin{array}{ll}
        \tau_{\leq d-1}((1-t^2)P_{\Z_{\sigma}}(t)) & \mbox{if } d\geq 3 \\
        1 & \mbox{if } d \leq 2. 
    \end{array}
\right.
$$

Here 
$$P_{\sigma}(t) = \sum_{j\geq 0}\dim_{\QQ}\, \H^{j-n}(IC_{\X})_{x_\sigma}t^j$$
is the local Poincar\'e polynomial, $d = \dim\, \sigma$, $x_\sigma$ is a point in the orbit associated with $\sigma$, and the corresponding
affine toric variety of $\sigma \in \Sigma$ is, after removing the torus factor, an affine cone over the projective toric variety $\Z_{\sigma}$ (for the
torus $(\CC^\star)^{d}$). Finally $\tau_{\leq d-1}$ stands for the truncation of polynomials to degrees $\leq d-1$. These two formulae
recursively define the \emph{$h$-numbers} $h_{i}(\Sigma_{\X}):= \dim_{\QQ}\, IH^{2i}(\X;\QQ)$. In particular, $$\dim_{\QQ}\, IH^{2i+1}(\X;\QQ) = 0\text{ for any }i, \text{ and thus } P_{\X}(t)= h(\Sigma; t^2) := \sum_{i= 0}^{n}h_{i}(\Sigma)t^{2i}.$$

We now pass to the definition of the $g$-polynomials. The projective
toric variety $Z_{\sigma}$ corresponds to a polytope $Q = Q_{\sigma}$ in the sense
that the defining fan of $Z_{\sigma}$ is the normal fan $\Sigma_{Q}$ associated with $Q$.
For any
$1 \leq i\leq d$, we let $h_{i}(Q):= h_{i}(\Sigma_{Q})$. We define the \emph{$g$-invariants} of $Q$ by setting
$g_{0}(Q) := h_{0}(Q)$ and $g_{i}(Q)= g_{i}(\sigma) :=  h_{i}(Q) - h_{i-1}(Q)$ for $1\leq i \leq \lfloor d/2 \rfloor $ and $d\geq 3$. If $d\leq 2$, then we let $g_{0}(Q) =1$
and $g_{i}(Q) = 0$ for $i>0$. In particular, we see from above that
the \emph{$g$-polynomial}  $g(Q; t^{2}) = \sum_{i =0}^{\lfloor d/2 \rfloor}g_{i}(Q)t^{2i}$ coincides with the local Poincar\'e polynomial $P_{\sigma}(t)$.
For any face $E$ of $Q$, there is a polytope $Q/E$ such that the poset of its faces is isomorphic to the 
poset of faces of $Q$ which contain $E$ (see \cite{BM99}). Moreover, according to \cite[Section 2, Proposition 2]{BM99},
if $F$ is a face of $E$, then there exist unique polynomials $g(E,F; t^{2})$ such that 
$$g(\sigma; t^{2}) = g(Q; t^{2}) = \sum_{F\prec E\prec Q} g(E, F; t^{2})g(Q/E;  t^{2}),$$
where here $\prec$ is the face relation. 
Coefficients of $g(E, F, t^{2})$ are denoted by $g_{i}(E, F)$ and are usually called  \emph{relative $g$-invariants}.  Observe that $g(Q, t^{2}) = g(Q, Q; t^{2})$. As there is
a one-to-one correspondence between faces of $\sigma$ and faces of $Q$,
a similar notation for relative $g$-invariants with respect to
the faces of $\sigma$ will be used.
\\

We now recall the main result in \cite{AL17}. In this paragraph, we consider a normal $\TT$-variety $X$ of complexity one  associated with a 
divisorial fan $\EE$ over $(Y, N)$. We assume moreover that $X$ is contraction-free (i.e., all the loci of elements of $\EE$ are affine) and complete. We denote by  $\Sigma(\EE)$ the smallest fan containing all the tails of each element of $\EE$. To every point $y\in Y$ we consider the fan $\EE_y$ of $N_\QQ\times \QQ$ generated by the cones
$$C_{y}(\D^{i}):= \cone((\sigma_i\times \{0\}) \cup (\D^{i}_{y}\times \{1\}))\text{ and }\cone((\sigma_i\times \{0\})\cup (\sigma_i\times \{-1\}))$$ 
for any $i\in I$. Also, let $\supp(\EE) = \{y\in Y\,|\, \D_{y}^i \neq \sigma_{i}\text{ for some }i\in I\}$. 

The following result determines the intersection cohomology Betti numbers in the contraction-free case. This was made possible thanks to the work of de Cataldo, Migliorini, and Musta\c{t}\u{a} \cite{CMM15} on the topology
of toric fibration (see \cite[Theorem 1.2]{AL17}\footnote{In the statement of \cite[Theorem 1.2]{AL17} (Page 166, Line 1)
the expression  ``perverse sheaves'' is inappropriate and should be replaced by ``constructible complexes of sheaves''.}). 
Remark that another way to compute it, this time in terms of the topology of the fibers of the quotient map, might be to use the results in \cite{CML08}.
\begin{theorem}\label{t-contfree}\cite[Theorem 1.1]{AL17}
Let $X$ be a complete contraction-free normal $\TT$-variety with defining
divisorial fan $\EE$ over the smooth complete curve $Y$. Let $\rho_{g}(Y)$ be the genus of $Y$ and let $r$ be the cardinality
of the finite set $\supp(\EE)$. Then  
$$P_{X}(t) = ((1-r)t^{2} + 2\rho_{g}(Y)t + 1-r)h(\Sigma(\EE); t^2) + \sum_{y\in \supp(\EE)}h(\EE_{y}; t^{2}).$$ 
\end{theorem}
\begin{remark}\label{r-contfree}
Note that \cite[Theorem 1.1]{AL17} is only stated in the case where $X$ is projective. Nevertheless, the proof of 
\emph{loc. cit.} extends to the complete case. Indeed, the usual Betti numbers of smooth complete varieties can
be described via $E$-polynomial methods. Thus, Theorem \ref{t-contfree} is obtained by adapting Proposition 3.2 and Lemmata 4.4, 4.5, 4.6 of \cite{AL17} in terms of $h$-invariants of complete fans. 
\end{remark}
\begin{example}\label{e-surface2}\emph{Returning to the case of complete $\CC^{\star}$-surfaces} (see also \cite{FK86}).
Assume that $X$ is a complete normal $\CC^{\star}$-surface given by a divisorial fan
over the complete curve $Y$. Let $\tilde{\EE}$ be the divisorial fan of the contraction space $\tilde{X}$. By combining Example \ref{ex-surf}, Theorem \ref{t-contfree}, and \cite[Section 1, Remark iii)]{Fie91} we get the concrete formula
$$P_{X}(t) = ((1-r)t^{2} +  2\rho_{g}(Y)t + 1-r)(t^{2} + 1) - |E|t^{2}$$ 
$$+ \sum_{y\in \supp(\EE)}(t^{4} + (\delta(\tilde{\EE}_{y})-2)t^{2} + 1),$$
where $\delta(\tilde{\EE}_{y})$ is the number of rays of the fan $\tilde{\EE}_{y}$ and the set $E\subseteq X$ is the image of the exceptional locus of the contraction map.
Actually, $|E|$ coincides with the number of polyhedral divisors in $\EE$
defined in the whole curve $Y$ and having maximal tail.
For instance, let us take the divisorial fan $\mathcal{N}$ over $\P^{1}_{\CC}$ generated by $\{\D^{1}, \D^{2}, \D^{3}\}$, where 
$$\D^{1} = \sum_{y\in \P^{1}_{\CC}\setminus\{\infty\}}\D^{1}_{y}\cdot [y] \text{ with }\D^{1}_{0} = [-1, 0]\text{ and }\D^{1}_{y} = \{0\}\text{  for any }y \in \P^{1}_{\CC}\setminus\{0, \infty\};$$
$$\D^{2} = \sum_{y\in \P^{1}_{\CC}}\D^{2}_{y}\cdot [y] \text{ with } \D^{2}_{\infty} = [1/2, \infty[  \text{ and }   \D^{2}_{y} = [0, \infty[\text{  for any }y \in \P^{1}_{\CC}\setminus\{\infty\};$$
$$\D^{3} = \sum_{y\in \P^{1}_{\CC}}\D^{3}_{y}\cdot [y] \text{ with }\D^{3}_{0} = ]-\infty, -1], \D^{3}_{\infty} = ]-\infty, 1/2]\text{ and }\D^{3}_{y} = ]-\infty, 0]$$
for any $y \in \P^{1}_{\CC}\setminus\{0, \infty\}.$
Then, denoting by $X(\mathcal{N})$ the complete $\CC^{\star}$-surface associated with $\mathcal{N}$, we obtain from above that 
$$P_{X(\mathcal{N})}(t) = ((1-2)t^{2} + 1-2)(t^{2} + 1)  - |E|t^{2} +  \sum_{y \in \{0, \infty\}}(t^{4} + (\delta(\tilde{\mathcal{N}}_{y})-2)t^{2} + 1).$$
Note that $\delta(\tilde{\mathcal{N}}_{0}) = 5$, $\delta(\tilde{\mathcal{N}}_{\infty}) = 4$, and $|E| = 2$.	Therefore $P_{X(\mathcal{N})}(t) = t^{4} + t^{2} + 1$. The surface $X(\mathcal{N})$ is in fact isomorphic to $\P^{2}_{\CC}$  (compare \cite[Example 2.2.3]{AKP15}).
\end{example}
\subsection{Isolated attractive fixed points} \label{sec-isol}
For a variety $X$ equipped with an action of the torus $\TT$, a fixed point $x\in X^{\TT}$ is \emph{attractive} if all the weights for the  $\TT$-action on the tangent space $T_{x}X$ are contained in an open half space. We refer to \cite[Proposition A2]{Bri99} for other characterizations of the notion of attractive fixed point.
As a by-product, we give a formula for the intersection cohomology Betti numbers of any normal affine $\TT$-variety of complexity one 
having a unique attractive fixed point. We start with a known lemma yielding a comparison between global and local
intersection cohomologies around an attractive fixed point locus. 

\begin{lemma}\cite[Lemma 6.5]{DL91}\label{l-localaff}
Let $X$ be a normal affine  $\TT$-variety of complexity one with a unique attractive fixed point $x$. Then 
  $$IH^j(X; \QQ) \simeq \H^{j-\dim\, X}(IC_{X})_{x}\text{ holds for any }j\in \ZZ.$$
Furthermore, let $\tilde{X}= \tilde{X}(\D)$ be a normal $\TT$-variety of complexity one  defined as the relative spectrum of the sheaf of a polyhedral divisor $\D\in\PPDiv(Y, \sigma)$ over a smooth projective curve $Y$,
where the tail $\sigma$ is assumed to be full-dimensional. Take $u\in N$ in the relative interior of the tail cone $\sigma$ of $\D$. 
In this way, $\CC^{\star}$ acts on $\tilde{X}$ via the one-parameter subgroup $\lambda^u: \CC^{\star}\rightarrow \TT$. 
Choose $d\in \ZZ_{>0}$ so that $\tilde{X}_{d} =  \tilde{X}/\mu_d(\CC)$ (where  $\mu_{d}(\CC)\subseteq \CC^{\star}$ acts by restriction) has stabilizers of the form $\{1\}$ or $\CC^{\star}$. 
Then  $$ \tilde{X}^{\TT}\simeq \tilde{X}_{d}^{\CC^{\star}}\simeq Y$$
and we have the isomorphisms
$$IH^j(\tilde{X}; \QQ) \simeq \mathbb{H}^{j}(\tilde{X}_d, \iota^{\star}IC_{\tilde{X}_d}[-\dim\, \tilde{X}_d]),$$
where $\iota : \tilde{X}^{\CC^{\star}}_{d}\hookrightarrow \tilde{X}_{d}$ is the inclusion.  
\end{lemma}

\begin{proof}
 The first assertion is a direct consequence of \cite[Lemma 4.2]{CMM15}, which is an adaptation of \cite[Lemma 6.5]{DL91} for complex numbers. For the second assertion, we proceed as follows.
In view of  \cite[Lemma 2.12]{Kir86}, we may write $\tilde{X}$ for $\tilde{X}_d$ where $d\in \ZZ_{>0}$  is chosen as in
the statement.  By assumption, the action $\CC^{\star}\times \tilde{X}\rightarrow \tilde{X}$ extends to a morphism
$h: \AA^{1}_{\CC}\times \tilde{X}\rightarrow \tilde{X}$ such that 
$$h^{-1}( \tilde{X}^{\CC^{\star}}) = (\{0\}\times \tilde{X})\cup (\AA^{1}_{\CC}\times  \tilde{X}^{\CC^{\star}})$$
and the morphism
$$\CC^{\star}\times \tilde{X}_{0}\rightarrow \CC^{\star}\times \tilde{X}_{0}, \, (t, z)\mapsto (t, t\cdot z)$$
is an automorphism, where $ \tilde{X}_{0}$ is the set $\tilde{X}\setminus \tilde{X}^{\CC^{\star}}$. Considering the inclusion
$j: \tilde{X}_{0}\rightarrow \tilde{X}$, by adjunction, we only need to show that 
$$\mathbb{H}^{\star} (\tilde{X}, j_{!}(IC_{\tilde{X}_{0}})) = 0.$$
Let $g:\tilde{X}\rightarrow \AA^{1}_{\CC}\times \tilde{X}$ be defined as $g(x) = (1, x)$ so that $h\circ g = {\rm id}_{\tilde{X}}$ and therefore the composition
$$\mathbb{H}^{\star}(\tilde{X}, j_{!}(IC_{\tilde{X}_{0}}))\rightarrow \mathbb{H}^{\star}( \AA^{1}_{\CC}\times\tilde{X}, h^{\star}j_{!}(IC_{\tilde{X}_{0}}))\rightarrow \mathbb{H}^{\star}(\tilde{X}, g^{\star}h^{\star}j_{!}(IC_{\tilde{X}_{0}}))$$
is an isomorphism. Hence it suffices to have that  $\mathbb{H}^{\star}(\tilde{X}, h^{\star}j_{!}(IC_{\tilde{X}_{0}})) = 0$.
Let $a: \CC^{\star}\times \tilde{X}_{0}\hookrightarrow \AA_{\CC}^{1}\times \tilde{X}$ be the inclusion.
Note that we have a Cartesian commutative diagram 
 $$\xymatrix{
     \CC^{\star}\times \tilde{X}_{0}  \,\,\ar[r]^{h_{0}} \ar[d] ^{a} & \tilde{X}_{0} \ar[d]^{j}\\     \AA^{1}_{\CC}\times\tilde{X} \,\, \ar[r]^{h} & \tilde{X}},$$  
where $h_{0}: \CC^{\star}\times \tilde{X}_{0}\rightarrow \tilde{X}_{0}$ is induced by the $\CC^{\star}$-action.
So by base change (see \cite[Theorem 2.3.26]{Dim04}) we have $h^{\star}j_{!}\simeq a_{!}h_{0}^{\star}$. Moreover, as $\CC^{\star}$ acts on the 
complex of sheaves $IC_{\tilde{X}_{0}}$, we get that 
$$ h^{\star}_{0}(IC_{\tilde{X}_{0}})\simeq  {\rm proj}_{2}^{\star}(IC_{\tilde{X}_{0}})\simeq IC_{\CC^{\star}\times\tilde{X}_{0}}[-1]\simeq \QQ_{\CC^{\star}}\boxtimes IC_{\tilde{X}_{0}}.$$
Therefore it follows that 
$$h^{\star}j_{!}(IC_{\tilde{X}_{0}})\simeq a_{!}h_{0}^{\star}(IC_{\tilde{X}_{0}})\simeq  a_{!}(\QQ_{\CC^{\star}}\boxtimes IC_{\tilde{X}_{0}}))\simeq j^{\sharp}_{!}(\QQ_{\CC^{\star}})\boxtimes j_{!}(IC_{\tilde{X}_{0}}),$$
where $j^{\sharp}:\CC^{\star}\hookrightarrow \AA^{1}_{\CC}$ is the inclusion. Since $\mathbb{H}^{\star}(\AA_{\CC}^{1}, j^{\sharp}_{!}(\QQ_{\CC^{\star}})) = 0$, we conclude, using the K\"unneth's formula that  $\mathbb{H}^{\star}(\tilde{X}, h^{\star}j_{!}(IC_{\tilde{X}_{0}})) = 0$, as required.
\end{proof} 
\begin{example}
Let us illustrate Lemma \ref{l-localaff} with a concrete example. We consider the blow-up $\tilde{X}$ of $\AA^{2}_{\CC} =  \CC^2$ at the origin. In terms of polyhedral divisors, we are looking at $\D$ over $\P^{1}_{\CC}$, the lattice $N$ is $\ZZ$, and we have the conditions $$\D = \sum_{y\in \P^{1}_{\CC}}\D_{y}\cdot [y]\text{ with }\D_y = \QQ_{\geq 0}\text{ if }y\neq 0\text{ and }\D_0 =  \{1\} + \QQ_{\geq 0}.$$ The contraction map $\pi: \tilde{X}\rightarrow X$ for $\D$ is the blow-up map. If $V\simeq \P^{1}_{\CC}$ is the exceptional locus, then the pair $(V, \tilde{X})$ is locally isomorphic for the Zariski topology to the pair $(\P^{1}_{\CC}\times\{0\}, \P^{1}_{\CC}\times \AA^{1}_{\CC})$ and therefore (compare with \cite[5.3.1]{DL91}), we have
$IC_{\P^{1}_{\CC}}\simeq IC_{V}\simeq IC_{\tilde{X}}[-1]_{|V}$. Note that $$\tilde{X}^{\CC^{\star}} = V.$$ So by Lemma  \ref{l-localaff}, we get from the equalities
$$\mathbb{H}^i(V, (IC_{\tilde{X}})_{|V}) =  \mathbb{H}^{i+1}(\P^{1}_{\CC},IC_{\P^{1}_{\CC}}),$$ the following expressions
$$H^{j}(\tilde{X}; \QQ)  =   IH^{j}(\tilde{X}; \QQ)  \simeq \mathbb{H}^{j-2}(V, (IC_{\tilde{X}})_{|V})\simeq \mathbb{H}^{j-1}(\P^{1}_{\CC},IC_{\P^{1}_{\CC}})\simeq H^{j}(\P^{1}_{\CC}; \QQ),$$ 
which is the excepted result since $\tilde{X}$ is the total space of the line bundle $\mathcal{O}_{\P^{1}_{\CC}}(-1)$.
\end{example}

Let  $X = X(\D)$ be a normal affine $\TT$-variety of complexity one with a unique attractive fixed point. 
This condition means that  $\D$ is defined over a smooth projective curve $Y$ and that the dual of its tail $\sigma^\vee$ is strictly convex. Choose $u\in N$ a primitive element in the relative interior of $\sigma$ and consider the embedding $\lambda^{u}:\CC^\star\hookrightarrow \TT$ attached to it. From this, after dividing by a finite group action, the variety $X(\D)$ becomes an affine cone over a projective $(\CC^{\star})^{n-1}$-variety $X(\EE^{u})$ that we now exhibit (see \cite[Section 5]{AH08}, \cite[Section 4, Remark]{IS11}). Denote by $M = M'\times \ZZ$, $N = N'\times \ZZ$ the decomposition constructed by the torus complement of the image of the embedding $\lambda^u$. For any $v\in N_{\QQ}'\times \QQ$ write $\pi_{1}(v)\in N_{\QQ}'$ and $\pi_{2}(v)\in \QQ$ for the first and the second components, and for every $y\in Y$, consider the map $\theta_{y}: N_{\QQ}'\rightarrow \QQ$ defined via the relation $-\theta_{y}(v) = \min \pi_{2}(\pi_{1}^{-1}(v)\cap \D_{y})$. The divisorial fan $\EE^{u}$ on $(Y, N')$ is the one arising from
the piecewise affine map $$N_{\QQ}'\rightarrow \Ca(Y), v\mapsto \sum_{y\in Y}\theta_{y}(v)\cdot [y].$$ 
In other words, if $\EE^u = \{\D^{u, i}\,|\, i\in I\}$, then the faces of the polyhedra  $\D^{u, i}_{y}$
are loci where the map $\theta_y$ is affine. The affine property on $\D^{u, i}_{y}$ means 
that there are $c_{y, i}\in \QQ$ and $m_{y,i}\in M_{\QQ}'$ such that for 
any $v\in \D^{u, i}_{y}$ we have that
$\theta_{y}(v) = c_{y,i} + \langle m_{y, i}, v\rangle.$ 
More precisely, by \cite[Proposition 5.1]{IV13} the subdivision generated by $\{\D^{u, i}_{y}\,|\, i\in I\}$
is equal to the coarsest polyhedral subdivision containing all the $\pi_{1}(\Delta)$,
where $\Delta$ runs over the faces of $\D_{y}$. Moreover, we add the following rule. Each polyhedral divisor
$\D^{u, i}$ with maximal tail $\sigma_{u, i}$ and with the condition that $ (\theta_{y})_{|\sigma_{u,i}}(0)$ is principal (i.e., the sum $\sum_{y\in Y} c_{y,i}\cdot [y]$ is a
principal divisor) is defined over the whole curve $Y$. This also applies for those polyhedral divisors
having their tail intersecting $\deg\, \D^{u,i}$ (the other being defined over an affine open subsets of $Y$). 
\begin{proposition}\label{t-fixedpointt}
Let $X(\D)$ be an affine $\TT$-variety of complexity one corresponding to
a polyhedral divisor $\D\in \PPDiv(Y, \sigma)$ over a smooth projective curve $Y$. Assume that $X(\D)$ has a unique attractive fixed point and set $n = \dim\, \TT$. Consider, as previously,  
a presentation where $X(\D)/\mu_{d}(\CC)$ is an affine cone over $X(\EE^{u})$ and assume that $n\geq 2$. Then 
$$P_{X(\D)}(t) = \tau_{\leq n}((1-t^2) P_{X(\EE^u)}(t)),$$
where $\tau_{\leq n}$ is the truncation to degrees $\leq n$.
\end{proposition}
\begin{proof}
As $\mu_{d}(\CC)$ acts by restriction of the $\CC^{\star}$-action, the global and local intersection cohomologies are unchanged by the quotient of $\mu_{d}(\CC)$ (see \cite[Lemma 2.12]{Kir86}).
So we may assume that $X(\D)$ is an affine cone over $X(\EE^{u})$. We conclude by combining Lemma  \ref{l-localaff} and \cite[Lemma 2.1]{Fie91}.
\end{proof} 
\begin{example}\label{e-elliptic}\emph{Affine $\CC^{\star}$-surfaces with a unique attractive fixed point.} Let $X = X(\D)$
be a normal affine $\CC^{\star}$-surface with polyhedral divisor $\D$ defined over a smooth projective 
curve $Y$. Using the argument of the proof of  Proposition \ref{t-fixedpointt}, we have that
$$P_{X}(t) = \tau_{\leq 1}((1-t^{2})P_{Y}(t)) = 1 + 2\rho_{g}(Y)t.$$
\end{example}
\begin{example}\emph{Affine $(\CC^{\star})^{2}$-threefolds with a unique attractive fixed point.}\label{e-aff3fold}
Let $X = X(\D)$
be a normal affine $(\CC^{\star})^{2}$-threefold with a unique fixed point, where $\D$
is a polyhedral divisor defined over the smooth projective curve $Y$.
Then the variety $X(\EE^{u})$ is a normal projective $\CC^{\star}$-surface 
and using Example \ref{e-surface2} and Proposition \ref{t-fixedpointt}, we have the formula 
$$P_{X}(t) = (1-r-|E^{u}|)t^{2} + 2\rho_{g}(Y)t + 1 + \sum_{y\in\supp(\tilde{\EE}^{u})}(\delta(\tilde{\EE}^{u}_{y})-3)t^{2},$$
where $X(\tilde{\EE}^{u})$ is the total space of the contraction map of $X(\EE^{u})$, and the set $E^{u}\subseteq X(\EE^{u})$ is the image of the exceptional locus. Note that we always have 
$|E^{u}|\leq 2$, and if $|E^{u}| =2$, then the definition of divisorial fans imposes that one of the $\delta(\tilde{\EE}^{u}_{y})$ is greater than $4$. This ensures that the coefficient in $t^{2}$ of the formula is non-negative.

Let us take a concrete example. Set $\sigma := \QQ_{\geq 0}^{2}\subseteq N_{\QQ} = \QQ^{2}$ and assume that $\D$
is the polyhedral divisor over $Y = \P^{1}_{\CC}$ defined by the conditions
$$\D_{0} = \left(\frac{1}{2},\frac{1}{2} \right) + \sigma,\,\, \D_{\infty}= [(0,1), (1, 0)] + \sigma, \text{ and } \D_{y} = \sigma \text{ for any } y\in \P_{\CC}^{1}\setminus\{0, \infty\}.$$ Then, in this case, the variety $X(\EE^{u})$ for $u = (1,1)$ is obtained by gluing the relative spectra
of the sheaves associated with the polyhedral divisors 
$$\D^{u, 1} = \sum_{y\in \P^{1}_{\CC}\setminus\{0\}}\D^{u, 1}_{y}\cdot [y] \text{ with }\D^{u, 1}_{\infty} = [-1, 1]\text{ and }\D^{u, 1}_{y} = \{0\}\text{  for any }y \in \P^{1}_{\CC}\setminus\{0, \infty\};$$
$$\D^{u, 2} = \sum_{y\in \P^{1}_{\CC}}\D^{u, 2}_{y}\cdot [y] \text{ with } \D^{u,2}_{\infty} = [1, \infty[ \text{ and } \D^{u,2}_{y} = [0, \infty[\text{  for any }y \in \P^{1}_{\CC}\setminus\{\infty\};$$
$$\D^{u, 3} = \sum_{y\in \P^{1}_{\CC}}\D^{u, 3}_{y}\cdot [y] \text{ with }\D^{u, 3}_{\infty} = ]-\infty, -1] \text{ and } \D^{u,3}_{y} = ]-\infty, 0] \text{ for any }y \in \P^{1}_{\CC}\setminus\{0\}.$$
Using our formula, we arrive at $P_{X}(t) =  1 + (\delta(\tilde{\EE}^{u}_{\infty}) - 3)t^{2} = 2t^{2} + 1$, which can be also deduced by remarking that $X = X(\D)$ is isomorphic to
the $3$-dimensional affine toric variety associated with the cone 
$$\QQ_{\geq 0}(1,1,2) + \QQ_{\geq 0}(1,0,-1) + \QQ_{\geq 0}(0,1,-1) + \QQ_{\geq 0}(0,1,0) + \QQ_{\geq 0}(1,0,0)\subseteq \QQ^{3}.$$
\end{example}

\subsection{Case of a single closed orbit contained in any orbit closure}\label{sec-attob}
 In this section, we provide a formula for the intersection cohomology Betti numbers of any normal
affine $\TT$-variety $X$ of complexity one with the condition $\CC[X]^{\TT} = \CC$. These $\TT$-varieties
correspond to proper polyhedral divisors  over smooth projective curves. First, we start with the following definition.

\begin{definition}\label{def-handy} A polyhedral divisor $\D\in \PPDiv(Y, \sigma)$ with $Y$ a smooth projective curve is said to be \emph{decomposable}
if  for any $m\in \sigma^{\vee}\cap M$ the divisor $\D(m)$ is principal whenever $\min_{v\in \deg(\D)}\langle m, v\rangle = 0$.
\end{definition}
The next lemma justifies our terminology of decomposable polyhedral divisors.

\begin{lemma}\label{l-handy}
Let $Y$ be a smooth projective curve and let $\D\in \PPDiv(Y, \sigma)$. Let $O$ be the unique closed orbit of $X(\D)$ which is
 contained in any orbit closure. Assume that $\D$ is decomposable or equivalently that the stabilizers of the points of $O$ are connected.
 Then there exist  a lattice decomposition $N =  N_{0}\oplus N_{1}$, a strictly convex polyhedral cone $\sigma'\subseteq N_{0, \QQ}$
with dimension equal to the rank of $N_{0}$, a  polyhedral divisor $\D'  \in \PPDiv(Y, N, \sigma')$ such that $\D_{y}'\subseteq N_{0, \QQ}$ for any $y\in Y$  and a $\TT$-equivariant isomorphism $A(Y, \D)\simeq A(Y, \D')$ giving rise to an isomorphism of polyhedral divisors $\D\simeq \D'$. This latter means that
there exists a group morphism $\delta: M\rightarrow \CC(Y)^{\star}$,  an isomorphism of lattices $\psi:M\rightarrow M$ sending $\sigma^{\vee}\cap M$ to $(\sigma')^{\vee}\cap M$ and an automorphism $\varphi: Y\rightarrow Y$ such that 
$$\varphi^{\star}(\D(m)) =  \D'(\psi(m)) + {\rm div}(\delta(m))\text{ for any }m\in \sigma^{\vee}\cap M,$$
where $\varphi^{\star}$ is the corresponding pullback of $\QQ$-Cartier divisors. 
\end{lemma}
\begin{proof}
This follows from the isomorphism $X\simeq \TT_{O}\times X_{1}$ given by Lemma \ref{l-pointf} and the result \cite[Theorem 8.8]{AH06}.
\end{proof} 
\begin{definition}\label{def-pointed} Let $Y$ be a smooth projective curve and let $\D\in \PPDiv(Y, \sigma)$ be a decomposable polyhedral divisor. Even if it is defined up to isomorphism, we denote by $$\bar{\D}\in \PPDiv(Y, N_{0}, \sigma)$$
the induced polyhedral divisor from the decomposition of Lemma \ref{l-handy}
with ambient lattice $N_{0}$ and whose coefficients are equal to the ones 
of $\D'$. The polyhedral divisor $\bar{\D}$ will be called an \emph{associated pointed polyhedral divisor} of $\D$. We also call \emph{corank} of the decomposable polyhedral divisor $\D$ the rank
of the lattice $N_{1}$ in Lemma \ref{l-handy}. 
\end{definition}

\begin{theorem}\label{t-betbetbet}
Let $X = X(\D)$ be an affine $\TT$-variety, where $\D\in \PPDiv(Y, \sigma)$ is a polyhedral divisor over a smooth projective curve $Y$. Let $M_{0}$ be a sublattice of $M$ of finite index such that for any $m\in M_{0}$, the $\QQ$-divisor $\D(m)$ is principal
whenever $$\min_{v\in \deg(\D)}\langle m, v\rangle = 0.$$ Then the polyhedral divisor
$$\mathfrak{D}'\in \PPDiv(Y, {\rm Hom}_{\ZZ}(M_{0}, \ZZ), \sigma),$$ 
which is $\D$ with respect to the lattice $M_{0}$, is decomposable and we have the formula for the Poincar\'e polynomial
$$P_{X}(t) = (1+t)^{s} P_{X(\bar{\D}')}(t),$$
where $\bar{\D}'$ is the associated pointed polyhedral divisor of $\D'$ and $s$ is its corank (note that $P_{X(\bar{\D}')}(t)$
can be computed via Proposition \ref{t-fixedpointt}). Moreover, for any point $x$ lying in the unique closed orbit $O$ of $X$,
the local Poincar\'e polynomial of $X$ at $x$ is $P_{X(\bar{\D}')}(t)$. 
\end{theorem} 
\begin{proof}
As in the proof of Theorem \ref{t-main}, the local and global intersection cohomology Betti numbers do not change if we substitute
$M$ by $M_{0}$. From Lemma \ref{l-pointf}, we therefore have $$X/G\simeq (\CC^{\star})^{s}\times X(\bar{\D}'), \text{ where }G= {\rm Hom}(M/M_{0}, \CC^{\star})$$ and thus by the Kunneth's formula $P_{X}(t) = P_{X/G}(t) = (1+t)^{s}P_{X(\bar{\D}')}(t).$
The other assertion for the local intersection cohomology is straightforward.
\end{proof}
\begin{example}
We consider the polyhedral divisor $\D$ over $\P_{\CC}^{1}$ of Example \ref{conterex1}  with non-trivial coefficients $\D_{0} =  (1, \frac{1}{2}) + \sigma$, $\D_{\infty} =  (0, -\frac{1}{2})+\sigma$ and tail $\sigma =  \QQ_{\geq 0}\times \{0\}$, and the sublattice
$$M_{0} =  \{(m,n)\in \ZZ^{2}\, |\, n\in 2\ZZ\}\subseteq M = \ZZ^{2}.$$
Substituting $M$ by $M_{0}$, we get a decomposable polyhedral divisor $\D'$ such that we can take $\bar{\D}'$ over $\P_{\CC}^{1}$ with the following conditions. The tail of $\bar{\D}'$ is $\QQ_{\geq 0}$ and $\bar{\D}'$ has only one non-trivial coefficient which is the coefficient $\{1\}+\QQ_{\geq 0}$ at the origin.
Therefore $X(\bar{\D}')$ is isomorphic to the complex plane $\AA^{2}_{\CC}$ and from Theorem \ref{t-betbetbet}
we get $$P_{X(\D)}(t) =  (1+t) P_{\AA^{2}_{\CC}}(t)  =  1+t.$$ 
\end{example}

\subsection{Computing with $g$-invariants}\label{sec-g-inv}
We are working with the normal affine $\TT$-variety $\tilde{X}(\D)$ 
equal to the relative spectrum
of the sheaf of $\mathcal{O}_{Y}$-algebras $$\mathcal{A}= \bigoplus_{m\in\sigma^{\vee}\cap M}\mathcal{O}_{Y}(\D(m)),$$
where $\D$ is a polyhedral divisor over a smooth projective curve $Y$  having full-dimensional tail.
Our aim is now to determine the intersection cohomology Betti numbers
of $\tilde{X}(\D)$ in terms of $g$-invariants of certain polyhedral cones (see Theorem \ref{g-vectorth} for a precise statement). 

For every algebraic variety $S$ over $\CC$,
we will let $\mathcal{IC}_{S} := IC_{S}[-\dim\, S]$ so that $$IH^{j}(S;\QQ)= \mathbb{H}^{j}(S, \mathcal{IC}_{S})$$ for any index $j$ such that $0\leq j\leq 2\, \dim\, S$. We start by recalling a classical result.
\begin{lemma}\cite[Sections 4,6]{BM99}\label{l-BM99}\label{ll-BM99}
Let $\sigma\subseteq N_{\QQ}$ be a full-dimensional strictly convex  polyhedral cone
and consider the associated toric variety $X_{\sigma}$. Let $V(\tau)$ be the orbit closure of $X_{\sigma}$
corresponding to a face $\tau$ of $\sigma$.
Then
$$i_{\tau}^{\star}\mathcal{IC}_{X_{\sigma}}\simeq \bigoplus_{\tau \prec \beta\prec \sigma}\bigoplus_{j\geq 0}(i_{\beta})_{\star}\mathcal{IC}_{V(\beta)}^{\oplus g_{j}(\beta, \tau)}[-2j],$$
where $i_{\beta}: V(\beta)\rightarrow V(\tau)$ denotes the natural inclusion. Here the $ g_{j}(\beta, \tau)$'s are the relative $g$-invariants associated with the pair $(\beta, \tau)$ (see Section \ref{sec-sec-prem-prem}). 
\end{lemma}
Our idea is to generalize Lemma \ref{ll-BM99} in the context of complexity-one torus actions. 
We say that a constructible complex of sheaves $\mathcal{F}\in D^{b}_{c}(Z)$ on an algebraic variety $Z$ is \emph{pure} if there is an isomorphism $\mathcal{F}\simeq \bigoplus_{\alpha}(\iota_{\alpha})_{\star}\mathcal{IC}_{V_{\alpha}}(\mathscr{L}_{\alpha})[d_{\alpha}]$, where
$V_{\alpha}$ is an irreducible Zariski closed subvariety of $Z$, $\mathscr{L}_{\alpha}$ is a simple local system defined on a Zariski dense open subset of
the regular locus of $V_{\alpha}$, $d_{\alpha}\in \ZZ$, and $\iota_{\alpha}: V_{\alpha}\rightarrow Z$ stands for the natural inclusion. 
\begin{lemma}\label{l-pure}
Set $\Gamma := \tilde{X}(\D)^{\TT}$ and let $\iota:\Gamma\rightarrow \tilde{X}(\D)$ be
the inclusion. Then $\iota^{\star}\mathcal{IC}_{\tilde{X}(\D)}$ is pure.
\end{lemma}
Before proving Lemma \ref{l-pure}, one needs some preparations. 
\\

Let $\EE$ be a divisorial fan over $(Y, N)$ and assume that any element of $\EE$ has affine locus. For any
$y\in Y$, we denote by $\Sigma_{y}(\EE)$ the fan of $N_{\QQ}\times \QQ$ generated by the Cayley cones
$$C_{y}(\D) := {\rm Cone} ((\sigma\times\{0\})\cup (\D_{y}\times \{1\})),$$
where $\D$ runs over $\EE$. Let $\phi: U_{y}\rightarrow V_{0}$ be a surjective \'etale morphism such that 
$U_{y}\subseteq Y$ and $V_{0}\subseteq \AA^{1}_{\CC}$ are Zariski open subsets, $y\in U_{y}$, $0\in V_{0}$, $\phi^{-1}(0) = \{y\}$ and $\supp(\EE)\cap U_{y} = \{y\}$. Such a morphism $\phi$ exists since $Y$ is smooth.    
We denote by $\EE(\phi, y)$ (respectively, $\EE(\phi, y)_{\rm tor}$) the divisorial fan over $(U_{y}, N)$ (respectively, $(V_{0}, N)$) generated
by the polyhedral divisors
$$\D_{|U_{y}}:= \sum_{z\in U_{y}}\D_{z}\cdot [z]\text{ respectively, } \D_{\phi}:= \sum_{z\in U_{y}}\D_{z}\cdot [\phi(z)],$$
where $\D$ runs over $\EE$. Note that the $\TT$-variety $X(\EE(\phi, y)_{\rm tor})$ is Zariski open dense inside the toric
variety associated to the fan $\Sigma_{y}(\EE)$. 
\begin{lemma}\label{lem-etet}
We have a $\TT$-equivariant isomorphism
$$ X(\EE(\phi, y)_{\rm tor})\times_{V_{0}}U_{y}\simeq X(\EE(\phi, y))$$
and the natural morphism
$$\eta: X(\EE(\phi, y))\rightarrow X(\EE(\phi, y)_{\rm tor})$$
is an \'etale morphism. 
\end{lemma}
\begin{proof}
For the first claim, one needs to prove that the divisorial fan structures are the same on both sides. 
We may assume that $\EE(\phi, y)$ and $\EE(\phi, y)_{\rm tor}$ have only one element. Doing
this reduction, one concludes by comparing the discrete $\TT$-invariant valuations on the function field
that are centered in  $X(\EE(\phi, y)_{\rm tor})\times_{V_{0}}U_{y}$ and in $ X(\EE(\phi, y))$ (see \cite[Lemma 4.4]{LPR19}).
The last claim follows from the fact that \'etaleness is invariant under base change.
\end{proof}

\begin{proof}[Proof of Lemma \ref{l-pure}]
We follow the ideas developed in \cite[Sections 5, 6]{BM99} for the toric case
and adapt them for torus actions of complexity one.  Let $u\in N$ be a lattice vector belonging to the relative interior of the tail cone $\sigma$. Then $\CC^{\star}$ acts on $\tilde{X}(\D)$ via the one-parameter subgroup associated with $u$. 
For each $y\in Y$ we consider the fan $$\Sigma_{y} = \Sigma_{y}^{u}\text{ of }N_{\QQ}\times \QQ\text{ with support }C_{y}(\D) := {\rm Cone} ((\sigma\times\{0\})\cup (\D_{y}\times \{1\}))$$ defined as follows. It is spanned by the cones of the form $\tau +\QQ_{\geq 0}u$, where $\tau$ runs over all faces of $C_{y}(\D)$ that does not contained $\sigma$.

Let $\EE$ be a divisorial fan over $(Y, N)$ whose polyhedral divisors are defined over smooth affine curves, and such that for any $y\in Y$ the family $(C_{y}(\D'))_{\D'\in \EE}$ generates
the fan $\Sigma_{y}$. From the combinatorics of divisorial fans,
one may construct a natural $\TT$-equivariant birational proper morphism
$p:X(\EE)\rightarrow \tilde{X}(\D)$, which is an isomorphism over $\tilde{X}(\D)\setminus \Gamma$. Note that $Z:= p^{-1}(\Gamma)$ is a prime divisor of $X(\EE)$.

Let  $X(\EE(\phi, y))\subseteq X(\EE)$ be the Zariski open subset of Lemma \ref{lem-etet}, where by construction 
we have $\Sigma_{y}(\EE)  =  \Sigma_{y}$. Let $\alpha$ be the least common multiple of the lattice indices $[N: N\cap {\rm span}(\tau) + u\ZZ]$,
where $\tau$ runs through all the faces of the fans $\Sigma_{z}$ not containing $\sigma$,
and $z$ runs through all the points of $U_{y}$.
The lattice $$ N + (u/\alpha)\ZZ\subseteq N_{\QQ}$$ gives rise to a torus $\TT'$. We denote 
by $G_{\mu}$ the kernel of the natural map $\TT\rightarrow \TT'$. According to \cite[Proposition 16]{BM99}, the quotient 
 $X(\EE(\phi, y)_{\rm tor})/G_{\mu}$ is a 
line bundle over  $\eta(Z\cap X(\EE(\phi, y))/G_{\mu}$, where $\eta$ is the \'etale morphism
$$\eta: X(\EE(\phi, y))\rightarrow X(\EE(\phi, y)_{\rm tor})$$
of Lemma \ref{lem-etet}. In particular, $Z/G_{\mu}$ identifies with $Z$ and the pair
$(X(\EE)/G_{\mu}, Z)$ is locally isomorphic to $(\AA^{1}_{\CC}\times Z, \{0\}\times Z)$ for the \'etale topology.
Hence we get $$j^{\star}\mathcal{IC}_{X(\EE)}\simeq \mathcal{IC}_{Z}\text{ where } j: Z\rightarrow X(\EE)$$
 is the natural inclusion (compare \cite[5.3.1]{DL91}). Now consider the Cartesian commutative diagram
 $$\xymatrix{
      Z\ar[r]^{j} \ar[d] ^q & X(\EE) \ar[d]^{p}\\
      \Gamma \ar[r]^{\iota} & \tilde{X}(\D)}$$
where $q$ is the restriction of $p$ on the subset $Z$. We follow the proof of \cite[Theorem 10]{BM99}. Since $p$ and $q$ are proper morphisms, by 
base change (see \cite[Theorem 2.3.26]{Dim04}) we have that $$q_{\star} j^{\star}\mathcal{IC}_{X(\EE)}\simeq \iota^{\star}p_{\star}\mathcal{IC}_{X(\EE)}.$$ By the preceding step and using the decomposition theorem (see Theorem \ref{t-bbdtheodec}), the left-hand side of this identity is pure.  
Using again
the decomposition theorem, this time, for the map $p$, we observe that $\iota^{\star}\mathcal{IC}_{\tilde{X}(\D)}$ is a summand of $\iota^{\star}p_{\star}\mathcal{IC}_{X(\EE)}$ (as $p$ is birational) having a complement. Finally, we conclude the proof of the lemma by using \cite[Lemma 15]{BM99}, which implies that
 $\iota^{\star}\mathcal{IC}_{\tilde{X}(\D)}$ is pure.
\end{proof}

\begin{lemma}\label{l-restr-fix}
There is an isomorphism
$$\iota^{\star}\mathcal{IC}_{\tilde{X}(\D)} \simeq \bigoplus_{j\geq 0}\left(\mathcal{IC}_{\Gamma}^{\oplus g_{j}(\sigma)}[-2j]\oplus \bigoplus_{y\in \supp(\D)} (i_{y})_{\star}\mathcal{IC}_{\{y\}}^{\oplus g_{j}(C_{y}(\D), \sigma)}[-2j]\right),$$
where $\supp(\D) = \{y\in Y\,|\, \D_{y} \neq \sigma\}$, $C_{y}(\D) := {\rm Cone} ((\sigma\times \{0\}) \cup (\D_{y}\times \{1\}))$, $\Gamma := \tilde{X}(\D)^{\TT}$, and $i_{y}:\{y\}\rightarrow \Gamma$, $\iota:\Gamma\rightarrow \tilde{X}(\D)$ are the natural inclusions.
\end{lemma}
\begin{proof}
Since both sides are pure objects (see Lemma \ref{l-pure}), by \cite[Remark 2.5.3]{Dim04} it suffices to establish the isomorphism over Zariski open subsets that cover $\Gamma$.
We identify $\Gamma$ with the base curve $Y$. Let $y\in Y$. Lemma \ref{lem-etet} provides the
Cartesian commutative diagram 
$$\xymatrix{
      U_{y}\ar[r]^{i_{U_{y}}} \ar[d] ^\phi & X(\D_{|U_{y}}) \ar[d]^{\eta}\\
      V_{0} \ar[r]^{i_{V_{0}}} & X(\D_{\phi})}.$$
Consequently,
$$(\iota^{\star}\mathcal{IC}_{\tilde{X}(\D)})_{|U_{y}} \simeq i_{U_{y}}^{\star}\mathcal{IC}_{X(\D_{|U_{y}})} \simeq i_{U_{y}}^{\star}\eta^{\star}\mathcal{IC}_{X(\D_{\phi})}\simeq \phi^{\star}i_{V_{0}}^{\star}\mathcal{IC}_{X(\D_{\phi})}.$$
Here the second isomorphism is a consequence of the identification $$\eta^{\star}\mathcal{IC}_{X(\D_{|V_{0}}')}\simeq \mathcal{IC}_{X(\D_{|U_{y}})},$$ which is due to the fact that $\eta$ is \'etale (see 
\cite[Section 4.2]{BBD82} and \cite[Lemma 2.4]{AL17} for more details). The last isomorphism is obtained by remarking that we may compose the pull-back derived functors (compare \cite[Remark 2.3.8]{Dim04}).
Using Lemma \ref{l-BM99} and $X(\D_{\phi})$ is embedded into the toric variety  $X_{C_{y}(\D)}$ associated with $C_{y}(\D)$ , it follows that 
$$i_{V_{0}}^{\star}\mathcal{IC}_{X(\D_{\phi})}\simeq \bigoplus_{j\geq 0}\left(\mathcal{IC}_{V_{0}}^{\oplus g_{j}(\sigma)}[-2j]\oplus (i_{0})_{\star}\mathcal{IC}_{\{0\}}^{\oplus g_{j}(C_{y}(\D), \sigma)}[-2j]\right).$$
Note that the Zariski closure of the image of the fixed point locus of $X_{C_{y}(\D)}$ for the $\TT$-action by the \'etale map $\eta$ is exactly the $\TT\times \CC^{\star}$-orbit closure 
associated with the cone $\sigma\subseteq C_{y}(\D)$.
Again applying $\phi^{\star}$, we arrive at
$$(\iota^{\star}\mathcal{IC}_{\tilde{X}(\D)})_{|U_{y}}\simeq (\bigoplus_{j\geq 0}(\mathcal{IC}_{\Gamma}^{\oplus g_{j}(\sigma)}[-2j]\oplus \bigoplus_{y\in \supp(\D)} (i_{y})_{\star}\mathcal{IC}_{\{y\}}^{\oplus g_{j}(C_{y}(\D),\sigma)}[-2j]))_{|U_{y}}.$$ 
Finally, we conclude by observing that we may cover $\Gamma$ by Zariski open subsets of the form $U_{y}$. This finishes the proof of the lemma.
\end{proof}
The following result determines the Poincar\'e polynomials of certain relative spectra
associated to a polyhedral divisor defined over a smooth projective curve.
\begin{theorem}\label{g-vectorth}
Let $\D$ be a polyhedral divisor over the smooth projective curve $Y$ and assume that
the tail $\sigma$ of $\D$ is full-dimensional. Let $\tilde{X}(\D)$ be the relative spectrum 
of the sheaf associated with $\D$. Then
$$P_{\tilde{X}(\D)}(t) = (t^{2} + 2\rho_{g}(Y)t + 1 - r)\cdot g(\sigma; t^{2}) + \sum_{y\in \supp(\D)} g(C_{y}(\D); t^{2}),$$
where $r$  is the cardinality of $\supp(\D)$, $\rho_{g}(Y)$ is the genus of $Y$, and $$C_{y}(\D) := {\rm Cone} ((\sigma\times \{0\}) \cup (\D_{y}\times \{1\}))$$ for any $y\in Y$. Here $g(\sigma; t^{2})$ and $g(C_{y}(\D); t^{2})$ are respectively the $g$-polynomials associated with the cones $\sigma$ and $C_{y}(\D)$ (see Section \ref{sec-sec-prem-prem}).
\end{theorem}
\begin{proof}
The proof of theorem follows from Lemma \ref{l-restr-fix} (after dividing by a finite subgroup of $\TT$), from the isomorphisms 
$$IH^j(\tilde{X}(\D); \QQ) \simeq \mathbb{H}^{j}(\Gamma, \iota^{\star}\mathcal{IC}_{\tilde{X}(\D)})\text{ for any }j\in \ZZ,$$
provided by Lemma \ref{l-localaff}, and by the equality $$g(C_{y}(\D),\sigma; t^{2}) = g(C_{y}(\D); t^{2}) - g(\sigma; t^{2})$$ (see the comment after \cite[Section 2, Proposition 2]{BM99}).
\end{proof}
\subsection{Main result}\label{sec-main} 
In this section, we gather all our computations that we previously did in order 
to describe the multiplicities in the decomposition
theorem for the contraction map (Theorem \ref{t-hvector}).  
Before that, we spend some times on the combinatorics of posets (see \cite{Sta97} for details).  
Let $(\pos, \leq)$ be a finite poset and set 
$$\int(\pos) =\{(a,b)\in \pos^2\,|\, a\leq b\}.$$
By \emph{incidence algebra} $\inc(\pos)$ of the poset $\pos$ we mean the set of functions $\alpha: \int(\pos)\rightarrow \ZZ[t, t^{-1}]$ endowed with the convolution product 
$$(\alpha\star\beta)(a,b) = \sum_{a\leq c\leq b}\alpha(a,c)\cdot \beta(c, b).$$
This product is associative and admits an element $1$, which is the restriction of the Kronecker symbol to $\int(\pos)$. Observe further that (from the proof of \cite[Proposition 3.6.2]{Sta97}) the element $\alpha\in \inc(\pos)$
is invertible if and only if the image of the diagonal by $\alpha$  in $\pos^{2}$ is contained in $(\ZZ[t, t^{-1}])^{\times}$. 

We will use the following inversion formula. Let $\alpha\in \inc(\pos)$ be an invertible element and consider two 
maps $\varepsilon, \eta$ from $\pos$ to $\ZZ[t, t^{-1}]$ subject to the relation
$$\varepsilon(a) = \sum_{b \in \pos, b\leq a}\alpha^{-1}(b,a)\cdot\eta(b)\text{ for every }a\in \pos.\text{ Then }\eta(b) = \sum_{a \in \pos, a\leq b}\alpha(a,b)\cdot\varepsilon(a)$$ for every $b\in \pos$.
\\ 

We again get back to the situation of the introduction. We have a normal variety $X$ with a faithful $\TT$-action of complexity one attached to a divisorial fan $\EE$ over $Y$, the map $\pi: \tilde{X}\rightarrow X$ stands for the contraction map, and the set $E\subseteq \tilde{X}$ is the image of the exceptional locus of $\pi$. We also assume that $X$ is complete. We emphasize that since the action of the torus is of complexity one, the set $E$ has only a finite number of orbits. We will use the poset isomorphism 
$HF(\EE)\simeq \mathcal{O}(E)$ discussed at the end of Section \ref{sec-r-tvar}.
\\

For every $O\in \mathcal{O}(E)$, we have dealt in  Lemma \ref{l-structure} with 
the open subset $\tilde{X}_{O}$ and the set $X_{O}:= \pi(\tilde{X}_{O})$, which is an affine Zariski open subset of $X$. In terms of the combinatorial object $HF(\EE)$, the variety $\tilde{X}_{O}$
is nothing but the relative spectrum of the sheaf associated with $\D_{O}$, where the torus orbit $O\in \mathcal{O}(E)$ corresponds to the hyperface $C = C(\D_{O})$.
We now introduce certain generating functions which will play an important role 
later on.
\begin{definition}
For each pair $(O_{1}, O_{2})\in \int(\mathcal{O}(E))$ (that is, $O_{2}\subseteq \bar{O}_{1}$) we set 
$$R_{O_{1}, O_{2}}(t) = \sum_{i\in \ZZ}\dim_{\QQ}\mathcal{H}^{i}(IC_{\bar{O}_{1}})_{x_2}t^{i}\in \ZZ[t, t^{-1}],$$
where $x_{2}\in O_{2}$. Note that $R_{O_{1}, O_{2}}$ does not depend on the choice of the point $x_{2}$.
\end{definition}
\begin{lemma}\label{l-X12}
The map $$R: \int(\mathcal{O}(E))\rightarrow \ZZ[t, t^{-1}], \,\, (O_{1}, O_{2})\mapsto R_{O_{1}, O_{2}}(t)$$
is an invertible element of the algebra $\inc(\mathcal{O}(E))$.
Let 
$$U_{1,2}: = \bigcup_{O\text{ orbit  of }\bar{O}_{1},\, O_{2}\subseteq \bar{O}}O$$
and denote by $\kappa: X_{1,2}\rightarrow U_{1,2}$ the normalization map.
Then $X_{1,2}$ is affine and for any point $x\in \kappa^{-1}(O_{2})$ we have that
$$R_{O_{1}, O_{2}}(t) = \sum_{i\in \ZZ}\dim_{\QQ}\mathcal{H}^{i}(IC_{X_{1,2}})_{x}t^{i}.$$
\end{lemma}
\begin{proof}
The poset isomorphism $HF(\EE)\simeq \mathcal{O}(E)$ implies that $U_{1,2} =  X_{O_{2}}\cap \bar{O}_{1}$,
where $X_{O_{1}}$ is the affine variety defined in Lemma \ref{l-structure}. Since $U_{1,2}$ is an affine torus embedding, its normalization $X_{1,2}$ is an affine toric variety (compare \cite[Theorem 1.3.5 and Proposition 1.3.8]{CLS11}).  Note that the toric variety $X_{1,2}$ has a unique closed orbit $O_{1,2}$, the preimage $\kappa^{-1}(O_{2})$ is therefore equal to $O_{1,2}$, and the restriction $\kappa_{|O_{1,2}}: O_{1,2}\rightarrow O_{2}$ is an \'etale morphism. Since intersection cohomology is invariant under normalization (see the proof of \cite[Corollary 8.2.31]{HTT08}),
it follows that 
$$\mathcal{H}^{i}(IC_{\bar{O}_{1}})_{x_2}\simeq \mathcal{H}^{i}(IC_{U_{1,2}})_{x_2}\simeq \mathcal{H}^{i}(IC_{X_{1,2}})_{x},$$
where $x$ is a point of $O_{1,2}$. Indeed, take any Zariski open subset $U\subseteq X_{1,2}$ containing $x$ such that 
$\kappa^{-1}(\{x_{2}\})\cap U =  \{x\}$. Set $V =  \kappa(U)$ (which is a Zariski open subset of $U_{1,2}$). Then from \cite[Corollary 8.2.31]{HTT08} and using proper base change,
we arrive at 
$$\mathcal{H}^{i}(IC_{U_{1,2}})_{x_2}\simeq \mathcal{H}^{i}(i_{x_{2}}^{\star}IC_{V})\simeq \mathcal{H}^{i}(i_{x_{2}}^{\star}\kappa_{\star}IC_{U}) \simeq  \mathcal{H}^{i}(i_{x}^{\star}IC_{U})\simeq  \mathcal{H}^{i}(IC_{X_{1,2}})_{x},$$
where $i_{x}: \{x\}\rightarrow U$  and $i_{x_{2}}: \{x_{2}\}\rightarrow V$ are the natural inclusions. This concludes the proof of a statement of the lemma.

 In order to prove that $R$ is invertible, we need to 
have $R_{O_{1}, O_{2}}(t)\in (\ZZ[t, t^{-1}])^{\times}$ whenever $O_{1} = O_{2}$. So assume that $O_{1} = O_{2}$. In this case, $O_{2}$
is the open orbit of $\bar{O}_{1}$. Therefore
$$\mathcal{H}^{i}(IC_{\bar{O}_{1}})_{x_2} = \mathcal{H}^{i}(IC_{O_{2}})_{x_2} = \left\{
    \begin{array}{ll}
        1 & \mbox{if }  i = -\dim\, O_{2} \\
        0 & \mbox{if } i \neq -\dim\, O_{2}. 
    \end{array}
\right.$$
This gives $R_{O_{1}, O_{2}}(t) = t^{-\dim\, O_{2}}$, finishing the proof of our lemma.
\end{proof}  
\begin{remark}
Note that using the bijection $HF(\EE)\simeq \mathcal{O}(E)$ and the description
of torus orbits in \cite[Section 10]{AH06}, one can explicitly describe the toric variety $X_{1,2}$ in terms
of $\EE$. Indeed, using Lemma \ref{l-fiber}, one sees that the normalization of $X_{O_{2}}\cap\bar{O_{1}}$, which is $X_{1,2}$, is up 
to a finite covering the toric variety for the torus $\TT_{O_{1}}$ associated to the maximal tail of the hyperfaces of $C_{O_{2}}$ in
$HF(\EE)$ that contains $C_{O_{1}}$. Here $C_{O_{i}}$ is the element of $HF(\EE)$ associated with $O_{i}\in \mathcal{O}(\EE)$. Therefore, according to \cite{DL91, Fie91}, the Laurent polynomial $R_{O_{1}, O_{2}}(t)$ is a combinatorial object with respect to $\EE$.
\end{remark}

Our next result summarizes our computation for determining the
multiplicities of the (shifted) simple perverse sheaves involving in the decomposition 
theorem for the contraction map. In the statement of this theorem, we will consider  the notations employed in \ref{t-contfree}. Note
that our approach was inspired by \cite[Section 7]{CMM15}.
\begin{theorem}\label{t-hvector}
Let $X$ be a complete normal variety with a faithful $\TT$-action of complexity one.
Denote by $E\subseteq X$ the image of the exceptional locus of the contraction map, and let $\tilde{\EE}$
be the divisorial fan describing the contraction space $\tilde{X}$. Without loss of generality, we may and do assume (according to \ref{t-main} (i))
that the points of $E$ have connected stabilizers.
Then 
$$P_{X}(t) = ((1-r)t^{2} + 2\rho_{g}(Y)t + 1-r)h(\Sigma(\tilde{\EE}); t^2) 
+ \sum_{y\in \supp(\tilde{\EE})}h(\tilde{\EE}_{y}; t^{2}) $$ $$- \sum_{O\in \mathcal{O}(E)}S_{O}(t) P_{\bar{O}}(t)t^{\dim\, X - \dim\, O},$$
where $r$ is cardinality of $\supp(\tilde{\EE})$ and $S_{O}(t)= \sum_{b\in \ZZ}s_{b, O}t^{b}$ is the generating function of the multiplicities of Theorem \ref{t-main}. The Laurent polynomial $S_{O}(t)$ is obtained by the following inductive procedure. Let $\mathcal{O}(E)$ be the poset of torus orbits of $E$.
 Let $$\mathcal{O}(E)\rightarrow HF(\EE),\,\,O\mapsto C(\D_{O})$$
be the poset isomorphism considered at the end of Section \ref{sec-r-tvar}, where $\EE$ is a divisorial fan describing $X$.
For each $O\in \mathcal{O}(E)$, the polyhedral divisor  $\D_{O}$ is decomposable (see Definition \ref{def-handy})
and we  write $\bar{\D}_{O}$ for the associated pointed polyhedral divisor (see Definition \ref{def-pointed}). By \ref{g-vectorth} we have that
$$P_{\tilde{X}(\bar{\D}_{O})}(t) = (t^{2} + 2\rho_{g}(Y) + 1 - r)\cdot g(\sigma_{O}; t^{2}) + \sum_{y\in \supp(\bar{\D}_{O})} g(C_{y}(\bar{\D}_{O}); t^{2}),$$
where $\sigma_{O}$ is the tail of $\bar{\D}_{O}$, while by \ref{t-fixedpointt} the polynomial $P_{X(\bar{\D}_{O})}(t)$ is expressed in terms of the Poincar\'e polynomial of a combinatorially determined projective $\TT$-variety of complexity one  of smaller dimension. Finally, 
$$S_{O_{1}}(t) = \sum_{O\prec O_{1}}R^{-1}_{O, O_{1}}(t)\left(P_{\tilde{X}(\bar{\D}_{O})}(t) -P_{X(\bar{\D}_{O})}(t)\right)t^{-\dim\, X}$$
for any orbit $O_{1}\in \mathcal{O}(E)$, where $R^{-1}_{O, O_{1}}$ denotes the inverse 
of $R_{O, O_{1}}$ in the incidence algebra $\inc(\mathcal{O}(E))$ of the poset $\mathcal{O}(E)$.
\end{theorem}
\begin{proof}
Our starting point is the isomorphism
$$\pi_{\star}IC_{\tilde{X}}\simeq IC_{X}\oplus \bigoplus_{O\in\mathcal{O}(E)}\bigoplus_{b\in\ZZ}(\iota_{O})_{\star}IC_{\bar{O}}^{\oplus s_{b, O}}[-b]$$
of Theorem  \ref{t-main} provided by the decomposition theorem.
Taking the hypercohomology on both sides gives
$$P_{X}(t) = P_{\tilde{X}}(t) - \sum_{O\in \mathcal{O}(E)}S_{O}(t) P_{\bar{O}}(t)t^{\dim\, X - \dim\, O}.$$
Our first formula follows from Theorem \ref{t-contfree} applied to the complete contraction-free $\TT$-variety $\tilde{X}$. For expressing the generating function $S_{O}$ in our inductive process, we need to look at the stalks of the cohomology 
sheaves at a point lying in the image of the exceptional locus. More precisely, let $O_{2}\in \mathcal{O}(E)$ and pick a point $x_{2}\in O_{2}$. Then  
$$\dim_{\QQ} \mathcal{H}^{i}(\pi_{\star}IC_{X})_{x_{2}} = \dim_{\QQ} \mathcal{H}^{i}(IC_{X})_{x_{2}} + \sum_{O_{1}\prec O_{2}}\sum_{b\in\ZZ}s_{b, O_{1}}\dim_{\QQ}\mathcal{H}^{i - b}(IC_{\bar{O}_{1}})_{x_{2}}.$$
By Lemma \ref{l-structure}, there is a Zariski open subset $\tilde{X}_{O_{2}}\subseteq \tilde{X}$ containing $\pi^{-1}(O_{2})$ such that it decomposes as the product
$\TT_{O_{2}}\times\tilde{X}(\bar{\D}_{O_{2}})$. Let $X_{O_{1}}$ be the image of $\tilde{X}_{O_{2}}$ by the map $\pi$. Then $X_{O_{2}}\subseteq X$ is also a Zariski open subset
and we also have $X_{O_{2}} \simeq \TT_{O_{2}}\times X(\bar{\D}_{O_{2}})$. Therefore
replacing $\tilde{X}$ by $\tilde{X}_{{O}_{2}}$ and $X$ by $X_{O_{2}}$ and expressing the last formula in terms 
of Laurent polynomials we obtain from Lemma \ref{l-localaff} the formula
$$\left(P_{\tilde{X}(\bar{\D}_{O_{2}})}(t) -P_{X(\bar{\D}_{O_{2}})}(t)\right)t^{-\dim\, X} = \sum_{O_{1}\prec O_{2}}R_{O_{1}, O_{2}}(t)\cdot  S_{O_{1}}(t).$$
By virtue of Lemma \ref{l-X12}, the map 
$$R: \int(\mathcal{O}(E))\rightarrow \ZZ[t, t^{-1}], \,\, (O_{1}, O_{2})\mapsto R_{O_{1}, O_{2}}(t)$$ is invertible in the incidence algebra $\inc(\mathcal{O}(E))$. Thus,
we conclude by using the inversion formula mentioned in the beginning of Section
\ref{sec-main}.
\end{proof}
\subsection{The case of threefolds.} \label{subsec-dim3}
 We now illustrate
Theorem \ref{t-hvector} in the special case of a $3$-dimensional complete normal $\TT$-variety $X$ of
complexity one.  Note that here $\TT$ is the torus $(\CC^{\star})^{2}$ and,  as before, we choose a divisorial fan $\EE$ over $(Y, N)$ describing our threefold $X$. Without loss of generality, we may and do assume that the stabilizers of points of  the image $E$ of the exceptional locus of the contraction map are connected. Note that
the orbits of $E$ are of dimension $\leq 1$.  For each orbit of dimension $1$, the normalization of its Zariski closure in $X$ is isomorphic to the projective line. Keeping this in mind, we wish to determine the 
corresponding Laurent polynomial $S_{O}(t)$ for every element $O\in\mathcal{O}(E)$.
We consider two cases. 
\\

\emph{Case 1.} Let $O_{2} \in \mathcal{O}(E)$ be a torus orbit of dimension $1$.
Then, in this case, $X(\bar{\D}_{O_{2}}), \tilde{X}(\bar{\D}_{O_{2}})$ are affine $\CC^{\star}$-surfaces,
and by Theorem \ref{g-vectorth}, Lemma \ref{l-X12}, and Example \ref{e-elliptic}, it follows that $S_{O_{2}}(t) = 1$.
Indeed, by Theorem \ref{g-vectorth} one has
$$ P_{\tilde{X}(\bar{\D}_{O_{2}})}(t) =  t^{2} + 2\rho_{g}(Y) + 1$$ 
and by Example \ref{e-elliptic} one has
$$ P_{X(\bar{\D}_{O_{2}})}(t) =  1 + 2\rho_{g}(Y).$$
Therefore 
$$ t^{-1} = t^{-3}\left(P_{\tilde{X}(\bar{\D}_{O_{2}})}(t) -P_{X(\bar{\D}_{O_{2}})}(t)\right) = R_{O_{2}, O_{2}}(t) S_{O_{2}}(t).$$
But from  Lemma \ref{l-X12} we have $R_{O_{2}, O_{2}}(t) =  t^{-1}$, whence $S_{O_{2}}(t) = 1$.
\\

\emph{Case 2.}
Let $O_{2}\in \mathcal{O}(E)$ be a torus orbit of dimension $0$. Using the preceding case, we have on the one hand 
$$\sum_{O_{1}\prec O_{2}}R_{O_{1}, O_{2}}(t)\cdot  S_{O_{1}}(t) =  R_{O_{2}, O_{2}}(t) S_{O_{2}}(t) + \sum_{O_{1}\prec O_{2}, O_{1}\neq O_{2}}R_{O_{1}, O_{2}}(t)\cdot  S_{O_{1}}(t)$$

\begin{equation}\label{eq}= S_{O_{2}}(t) + \sum_{O_{1}\prec O_{2}, O_{1}\neq O_{2}} R_{O_{1}, O_{2}}(t) = S_{O_{2}}(t)  + \eta_{O_{2}}t^{-1},
\end{equation}
where $\eta_{O_{2}}:= |\{O_{1}\in \mathcal{O}(E)\,|\, O_{1}\prec O_{2}\text{ and } O_{1}\neq O_{2}\}|.$
On the other hand 
$$\sum_{O_{1}\prec O_{2}}R_{O_{1}, O_{2}}(t)\cdot  S_{O_{1}}(t) = q_{O_{2}}(t),$$
 where $$ q_{O_{2}}(t):= t^{-3} \left(P_{\tilde{X}(\bar{\D}_{O_{2}})}(t) -P_{X(\bar{\D}_{O_{2}})}(t)\right).$$ 
So from Equation \ref{eq} we have $S_{O_{2}}(t) = q_{O_{2}}(t)  - \eta_{O_{2}}t^{-1}.$
\\

\emph{Conclusion.} Denote by  $\mathcal{O}_{2}(E)$ (respectively,  $\mathcal{O}_{3}(E)$) the set orbits 
of $E$ of codimension $2$ (respectively, $3$) in $X$. Our analysis shows that 
$$\sum_{O\in \mathcal{O}_{2}(E)}S_{O}(t) P_{\bar{O}}(t)t^{\dim\, X - \dim\, O} =  |\mathcal{O}_{2}(E)|(t^{2} +1)t^{2}.$$
Moreover, using Poincar\'e duality, we must have from our computation:
$$\sum_{O\in \mathcal{O}_{3}(E)}S_{O}(t) P_{\bar{O}}(t)t^{\dim\, X - \dim\, O} =  \sum_{O\in \mathcal{O}_{3}(E)}(q_{O}(t)t^{3} - \eta_{O}t^{2}) =0.$$
Indeed,  observe that each polynomial $q_{O}(t)t^{3} - \eta_{O}t^{2}$  is of degree $\leq 2$ (use Example \ref{e-aff3fold} and  Theorem \ref{g-vectorth}). Therefore by Theorem \ref{t-hvector} and \cite[Section 1, Remark iii)]{Fie91} we obtain the following explicit result.
\begin{theorem}\label{theore-dim3}
Let $X =  X(\EE)$ be any complete normal threefold with a faithful action of $\TT= (\CC^{\star})^{2}$. Considering the notation set before,
we have the formula:
$$P_{X}(t) = ((1-r)t^{2} + 2\rho_{g}(Y)t + 1-r)(t^{4} + (\delta(\Sigma(\tilde{\EE}))-2)t^{2} + 1)$$ 
$$+ \sum_{y\in \supp(\tilde{\EE})}(t^{6} + (\delta(\tilde{\EE}_{y}) - 3)t^{4} + (\delta(\tilde{\EE}_{y}) - 3)t^{2} + 1) - |\mathcal{O}_{2}(E)|(t^{2} +1)t^{2},$$
where $r$ is cardinality of $\supp(\tilde{\EE})$, the numbers $\delta(\Sigma(\tilde{\EE})), \delta(\tilde{\EE}_{y})$ are respectively the numbers of rays of the fans $\Sigma(\tilde{\EE}), \tilde{\EE}_{y}$, 
 and $\mathcal{O}_{2}(E)$ is the set of orbits of $E$ of codimension $2$ in $X$.
\end{theorem}
Let us give a concrete example.  
\begin{example} \emph{A quadric threefold with a $2$-torus action.}
\begin{figure}[t]\label{pic76}
\includegraphics[width=8cm, height=1.5cm]{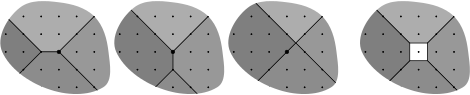}
\centering
\caption{}
\end{figure}
In this example, we assume that $N =  \ZZ^{2}$.
We take a divisorial fan $\EE$ over the projective line $\P^{1}_{\CC}$ from the classification of S\"uss in \cite[Section 5]{Sus14}.
The corresponding variety $X = X(\EE)$ is a smooth quadric threefold.
The support of each polyhedral divisor of $\EE$ is contained in the set $\{0, 1, \infty\}$ and $\EE$ is generated by
four polyhedral divisors $\D^{1}, \ldots, \D^{4}$ defined over $\P^{1}_{\CC}$:
\begin{itemize}
\item[(1)] The $\sigma_{1}$-polyhedral divisor $\D^{1}$ with 
 $$\D^{1}_{0} = \sigma_{1},\,\, \D^{1}_{1} =  {\rm Conv}((0,0), (0, -1)) + \sigma_{1},\,\, \D^{1}_{\infty} =  \left(\frac{1}{2}, \frac{1}{2}\right)+\sigma_{1}$$  and $\sigma_{1} = \QQ_{\geq 0}(1, 1) + \QQ_{\geq 0}(1, -1)$; 
\\

\item[(2)] The $\sigma_{2}$-polyhedral divisor $\D^{2}$ with 
$$\D_{0}^{2} =  {\rm Conv}((-1, 0), (0,0)) + \sigma_{2}, \,\, \D_{1}^{2} =  \{(0, -1)\} + \sigma_{2}, \, \,  \D^{2}_{\infty} =  \left(\frac{1}{2}, \frac{1}{2}\right)+\sigma_{2}$$
and $\sigma_{2} = \QQ_{\geq 0} (1, -1) + \QQ_{\geq 0} (-1, -1)$;
\\

\item[(3)] The $\sigma_{3}$-polyhedral divisor $\D^{3}$ with  
$$\D^{3}_{0} =  \{(-1, 0)\} +\sigma_{3}, \,\, \D_{1}^{3} = {\rm Conv}((0,0), (0, -1)) + \sigma_{3},\,\,  \D^{3}_{\infty} =  \left(\frac{1}{2}, \frac{1}{2}\right) +\sigma_{3}$$
and $\sigma_{3} = \QQ_{\geq 0} (-1, 1) +  \QQ_{\geq 0}(-1, -1)$;
\\

\item[(4)] The $\sigma_{4}$-polyhedral divisor $\D^{4}$ with  
$$\D_{0}^{4} =  {\rm Conv}((-1, 0), (0, 0)) + \sigma_{4}, \,\, \D_{1}^{4} = \sigma_{4}, \,\,  \D^{4}_{\infty} =  \left(\frac{1}{2}, \frac{1}{2}\right)+\sigma_{4}$$
and $\sigma_{4}=  \QQ_{\geq 0}(-1, 1) + \QQ_{\geq 0}(1, 1)$. 
\end{itemize}
The divisorial fan is illustrated in Figure 1 where the last diagram represents the degrees of the polyhedral divisors. 
Clearly, from the diagrams, the $\TT$-variety $X$ has $4$ fixed points, the Euler characteristic is $4$ and
using Lefchetz Hyperplane Theorem,  $P_{X}(t) =  1 + t^{2} + t^{4} + t^{6}$. Let us check this via Theorem 
\ref{theore-dim3}. First, we have 
$$r= 3,\,\,\delta(\tilde{\EE}_{0}) =  7,\,\, \delta(\tilde{\EE}_{1}) =  7,\,\, \delta(\tilde{\EE}_{\infty}) =  6,\,\, \delta(\Sigma(\tilde{\EE})) =  4,\text{ and }  |\mathcal{O}_{2}(E)| =  4.$$ 
Therefore we obtain
$$ ((1-r)t^{2} + 2\rho_{g}(Y)t + 1-r)(t^{4} + (\delta(\Sigma(\tilde{\EE}))-2)t^{2} + 1) =  -2(t^{6} + 3t^{4} + 3t^{2} +1),$$
$$\sum_{y\in \supp(\tilde{\EE})}(t^{6} + (\delta(\tilde{\EE}_{y}) - 3)t^{4} + (\delta(\tilde{\EE}_{y}) - 3)t^{2} + 1)
= 3t^{6} + 11t^{4} + 11t^{2} + 3 \text{ and }$$ 
$$- |\mathcal{O}_{2}(E)|(t^{2} +1)t^{2}= -4(t^{2}+t^{4}).$$
Hence adding all the contributions gives the expected result.  
\end{example}

{\em Acknowledgments.} 
We thank the referees for their careful remarks and for informing us about the mistake in  Lemma \ref{l-fiber} in a former version  of  the paper (see Remark \ref{rem-mistake} for more details).
We thank Javier Elizondo and Javier Fern\'andez de Bobadilla
for useful discussions. We also thank Mark Andrea de Cataldo and Simon Riche for email exchanges, and David Bradley-Williams for some corrections.
The research of the second author was conducted in the framework of the research training group
\emph{GRK 2240: Algebro-geometric Methods in Algebra, Arithmetic and Topology},
which is funded by the DFG. This article is supported by ERCEA Consolidator Grant 615655 - NMST and also by the Basque Government through the BERC 2014-2017 program and by Spanish Ministry of Economy and Competitiveness MINECO: BCAM Severo Ochoa excellence accreditation SEV-2013-0323. The second author was supported by the GRI CNRS Brazilian-French Network in Mathematics.

\end{document}